\documentclass[11pt,reqno]{amsart}

\RequirePackage[OT1]{fontenc}
\RequirePackage{amsthm,amsmath}
\RequirePackage[numbers]{natbib}

\usepackage{tikz}

\usepackage{amsmath}
\usepackage{amsthm}
\usepackage{latexsym}
\usepackage{amssymb}
\usepackage{amsfonts}
\usepackage{epsfig}
\usepackage{stackrel}
\usepackage{accents}
\usepackage{yfonts}

\usepackage{lineno}
\usepackage[backref=section,colorlinks=true,urlcolor=blue,citecolor=blue,pdfstartview=FitH]{hyperref}

\usepackage{psfrag}

\newtheorem{theorem}{Theorem}

\newtheorem{lemma}{Lemma}
\newtheorem{corollary}{Corollary}

\makeatletter
\newcommand*\bigcdot{\mathpalette\bigcdot@{.5}}
\newcommand*\bigcdot@[2]{\mathbin{\vcenter{\hbox{\scalebox{#2}{$\m@th#1\bullet$}}}}}
\makeatother

\theoremstyle{remark}
\newtheorem{remark}{Remark}

\renewcommand{\epsilon}{\varepsilon}

\def\expect{{\mathbb  E}}
\def\Pr{{\mathbb P}}
\def\real{{\mathbb  R}}
\def\nat{{\mathbb  N}}

\newcommand\ind[1]{{1}_{\{#1\}}}

\def\toP{\overset{\Pr}{\to}}

\def\sM{\mathtt{M}}
\def\sL{\mathtt{L}}
\def\sS{\mathtt{S}}
\def\sX{\mathtt{X}}
\def\sY{\mathtt{Y}}
\def\sZ{\mathtt{Z}}
\def\ss{\mathtt{s}}

\def\sz{\mathtt{z}}
\def\fn{\bar n}
\def\fk{\bar k}

\def\fsz{\bar \sz}
\def\fkappa{\bar \kappa}
\def\ff{\bar f}
\def\fss{\bar \ss}

\def\Linf{\hat L}

\usepackage{psfrag}
\usepackage{epsfig}

\begin{document}

\title{A Probabilistic Approach to 
Growth 
Networks}

\author{Predrag Jelenkovi\'c}
\address{Department of Electrical Engineering, Columbia University, New York, NY 10025}
\email{predrag@ee.columbia.edu}

\author{Jan\'e Kondev}
\address{Martin A. Fisher School of Physics, Brandeis University, Waltham, MA 02453}
\email{kondev@brandeis.edu}

\author{Lishibanya Mohapatra}
\address{Martin A. Fisher School of Physics, Brandeis University, Waltham, MA 02453}
\email{lishi87@brandeis.edu}


\author{Petar Mom\v cilovi\'c}
\address{Department of Industrial and Systems Engineering, Texas A\&M University, College Station, TX 77843}
\email{petar@tamu.edu}

\begin{abstract}
Widely used 
closed product-form networks have emerged recently as a primary model of stochastic growth of 
sub-cellular structures, e.g., cellular filaments. In the baseline model, homogeneous monomers attach and detach stochastically to individual filaments from a common pool of monomers, resulting in seemingly explicit product-form solutions. However, due to the large-scale nature of such networks, computing the partition functions for these solutions is numerically infeasible. To this end, we develop a novel 
methodology, based on a probabilistic representation of product-form solutions and large-deviations concentration inequalities, that yields explicit expressions for the marginal distributions of 
filament lengths. The parameters of the derived distributions can be 
computed from equations involving large-deviations rate functions, often admitting closed-form algebraic expressions. 
From a methodological perspective, a fundamental feature of our approach is that it 
provides exact results for order-one probabilities, even though our analysis involves
large-deviations rate functions, which characterize only vanishing probabilities on a logarithmic scale.


\end{abstract}

\keywords{Closed network, product-form solution, large-scale network, large deviations}

\date{\today}

\maketitle

%

%
%
%
\section{Introduction}


%

One of the central problems in cell biology is understanding how cells control the size of their
organelles, such as the cytoskeletal filaments, mitotic spindle and the nucleus. These organelles assemble in a pool of their building blocks and are dynamic -- their building blocks stochastically attach and detach until equilibrium is reached, and consequently the organelle reaches a particular size. Since the assembly and disassembly are governed by reversible chemical processes, reversible Markov chains are often used as natural models for organelle growth. Such models admit product-from solutions, which, although seemingly explicit, are infeasible to use for practically relevant problems because the evaluation of the normalization constant (partition function) is computationally intractable.

Outside of the emerging bio-molecular problems,
closed product-form models have been widely used for studying communication networks and systems since the 1980s~\cite{Pit79,MMi82,KnT90, Kog95, BBK99}, and more recently in analyzing coupon-based scrip trading~\cite{JSS14} and ride-sharing systems~\cite{BFL17,BDL19,BLW20}; an extensive list of references for non-biological applications can be found in~\cite{ADW13}. Communication systems are modeled as networks of queues, where infinite-server queue corresponds to the free monomer pool in our context, and single server queues are equivalent to filament models. The goal of designing proper communication systems is to minimize the latency and queue sizes, or equivalently limit the growth. 
Hence, overloaded (bottleneck) queues were of secondary interest in the prior literature, where the 
focus was on distributions of non-bottleneck queues, e.g., see~\cite{ADW13}. On the other hand, modeling 
large structures, i.e., the bottleneck queues, and understanding their distributional properties is the primary objective of our work. 
Another distinguishing feature of our work is that it considers a large number of bottleneck queues. 
From a mathematical perspective, the approximation approaches for closed product-form networks in the prior literature were primarily based on transform and analytical techniques, e.g., see~\cite{MMi82,KnT90,Kog95,BLW20,KBD20}.

In contrast to the prior approaches, we develop a fundamentally novel probabilistic methodology, which yields explicit, 
insightful and scalable characterization of the individual filament distributions. 
As stated in the abstract, the parameters of these distributions are the solution to 
equations involving large-deviations rate functions, which in some important cases reduce to solving 
a quadratic equation, e.g, see Corollaries~\ref{coro:homogen}~and~\ref{coro:twoTypes}.
Interestingly, when the number of bottlenecks is large enough, all of the queues, even the non-bottleneck ones,
can operate outside of the corresponding open network steady state.
Another  key characteristic of our solutions, even though they are computed from expressions 
involving large-deviations rate functions that describe vanishing probabilities on a logarithmic scale,
is that they yield in the limit the exact values for order-one probabilities. 
In other words, typical behavior of the filament size distribution is determined by large deviations 
of associated representation variables, which will be introduced.
Furthermore, our results provide greater insights into these systems since 
the probabilistic representations that we employ are easier to interpret.


In this work, we consider the limiting-pool mechanism which offers a simple explanation for how a structure growing in a pool of its building blocks can attain a particular size. It postulates that the rate of assembly of a structure is proportional to the number of free building blocks. Therefore, as the structure grows, the rate of assembly goes down and it stops growing once this rate is balanced by the rate of disassembly. There are several examples where it has been cited as the central mechanism controlling the size of organelles~\cite{GoH12}. 
Recently, it was experimentally demonstrated in~\cite{GVS13} that the size of mitotic spindle is proportional to the amount of cytoplasmic material, i.e., that the limited-pool mechanism indeed controls its size. Additionally, the study of scaling of the nucleolus in developing embryo of C. elegans is also in agreement with the limited-pool size control hypothesis~\cite{WeB15}.


It is important to note that, in all certainty, there are several other size-control mechanisms at play, which
can regulate the length of organelles, like the limiting-pool mechanism. 
For example, length-dependent feedback control mechanisms, e.g., antenna mechanism~\cite{MGK15}, has been 
postulated to regulate 
the size of cytoskeletal structures like actin cables and microtubule filaments~\cite{How11}. 
In these organelles, precise length control is often necessary as large changes in size can result in the loss of physiological properties of the organelles~\cite{CGY11}. For example cilia (organelles used for swimming in algae Chlamydomonas) cut to a smaller size grow back to their original size, 
thus implying that the length is closely monitored~\cite{GoH12,Mar05}. 
Furthermore, not only that the same type of structures can grow form a common pool of equal monomers, but also different organelles can be built from a shared pool of identical monomers, such as actin cables and actin patches in yeast, which are both assembled from actin monomers~\cite{MiD11}. For a recent overview of various proposed control mechanisms of organelle growth see~\cite{MGJ16}. In light of this, our analysis of the baseline limited-pool mechanism can be viewed as an important null hypothesis~\cite{Mar16} and our results can be used to possibly confirm or reject this hypothesis -- thereby suggesting the existence of more intricate control mechanisms in cells and in-vitro systems. The novel mathematical methodology developed here can be potentially extended to more elaborate control mechanisms as well.

{\em Notation.} For two functions $f$ and $g$, we write $f(x) \sim g(x)$ and $f(x) \ll g(x)$, as $x\to\infty$, to denote $f(x)/g(x) \to 1$ and $f(x)/g(x) \to 0$, as $x \to \infty$, respectively. The symbols $\vee$ and $\wedge$ represent the maximum and minimum operators, respectively; $\ind{\cdot}$ is the standard indicator function. For $n\geq k$, let $(n)_0:=1$ and 
\[
(n)_k := n (n-1) \cdots (n-k+1).
\]

{\em Methodology and summary of the results.} Our approach is based on the probabilistic representation for closed product-form networks proposed in~\cite{MLH17}. In the context of the baseline model, the high-level idea is to represent stationary distributions of filament lengths and the size of the free monomer pool in terms of expected values of functions involving independent random variables. For simplicity, consider a system with $m$ monomers and $f$ homogeneous filaments. The stationary distribution of the length $L$ of such filaments can be shown to satisfy (see Appendix~\ref{sec:pr}), for $l=0,1,\ldots$,
\begin{equation}
\frac{\Pr[L= l]}{\Pr[L=0]} = \frac{\expect [ (m-l+f -2 - \sM)_{f-2} \ind{\sM \leq m-l} ]}{\expect [ (m+f -2 - \sM )_{f-2} \ind{\sM  \leq m} ]},  \label{eq:introratio}
\end{equation}
where $\sM$ is a Poisson random variable with the mean defined by parameters of the system. Depending on the relation between $m$, $f$ and $\expect\sM$, several operational regimes arise. Our Theorem~\ref{thm:linear} covers the case when~$f$ is linear in the large number of monomers~$m$.
Here, the closed system deviates from the corresponding open model, in which the free monomer pool has a true Poisson distribution,  due to the large number of filaments relative to the size of the monomer pool. 
In this case, Lemma~\ref{lemma:prodX} indicates that the mass in expectations in~\eqref{eq:introratio} is concentrated around $m^{-1} \sM \approx \psi < m^{-1} \expect \sM$, where $\psi$ is a solution of an equation involving a large-deviations rate function~$\ell_{\sM}$ that describes the left tail of $\sM$. As a result of this, for large $m$, the asymptotic distribution of the filament length $L$ is geometric
\begin{align*}
\frac{\Pr[L = l]}{\Pr[L=0]} &\approx \frac{\expect [ (m-l+f -2 - \sM )_{f-2} \ind{m^{-1} \sM \approx \psi} ]}{\expect [ (m+f -2 - \sM)_{f-2} \ind{m^{-1} \sM \approx \psi} ]}  \\
&\approx \left( \frac{1 - \psi }{ m^{-1} f +1 - \psi} \right)^l;
\end{align*}
interestingly, the parameter of the geometric distribution is determined by large-deviations behavior of $\sM$, i.e., typical behavior of~$L$ is determined by large-deviations behavior of~$\sM$. 
We would like to emphasize that the presence of the ratio in the preceding expression is crucial 
since it allows for the "cancelation" of the difficult to compute terms in the numerator and denominator. 
Otherwise, evaluating the expectations separately by means of Varadhan's lemma~\cite[Section~2.2]{WES95} does not yield sufficient precision and the necessary extension, if possible at all, would require major technical work. 
 Indeed, the large-deviations rate function~$\ell_{\sM}$ characterizes the behavior on the logarithmic scale only, while the obtained distribution of~$L$ is asymptotically exact. Our second result, Theorem~\ref{thm:linear2}, covers the case when $f \ll m$ and $m^{-1} \expect\sM > 1$.
Now, the different behavior of the closed system compared to the corresponding open model 
arise from the relatively small number of monomers.
In this situation, the mass in the expectations in~\eqref{eq:introratio} is concentrated around the right boundary, and we show that, for large~$m$,
\begin{align*}
\frac{\Pr[L = l]}{\Pr[L=0]} &\approx \frac{\expect [ (m-l+f -2 - \sM )_{f-2} \ind{ \sM \approx m-l} ]}{\expect [ (m+f -2 - \sM )_{f-2} \ind{ \sM \approx m} ]}  \\
&\approx \frac{\Pr[\sM = m - l]}{\Pr[\sM= m]} \\
&\approx e^{l \, \ell'_{\sM}(1)},
\end{align*}
where $\ell'_{\sM}$ is the derivative of $\ell_{\sM}$; that is, the stationary distribution of $L$ is geometric in this case as well. 
Finally, our last result in Theorem~\ref{thm:ld3}, which is applicable when $1 \ll  f \ll m$ and $m^{-1} \expect\sM <1$,
connects Theorem~\ref{thm:linear} ($f$ linear in $m$) to the  finite filament case in~\cite{MLH17}.
In this case, $L$ increases with $m$ and an appropriately scaled version of $L$ is exponentially distributed.

{\em Organization.} The rest of the paper is organized as follows. In the next section, we introduce the baseline model and random variables associated with the probabilistic representation. Section~\ref{sec:ld} contains some preliminary results on large deviations. The main results are stated in Section~\ref{sec:results}. We illustrate a broader applicability of our approach in Section~\ref{sec:ext}. Details of the probabilistic representation, some technical results and proofs can be found in the appendices. 

\section{Model}
\label{sec:model}

Consider a closed system consisting of $m$ monomers and $f$ filaments. Each monomer has a rate $\lambda_i$, $1 \leq i \leq f$, of assembly to filament~$i$. Each filament disassembles one monomer at a time with disassembly rates $\mu_i$, $1 \leq i \leq f$. Define $\kappa_i := \mu_i/\lambda_i$, $1 \leq i \leq f$, as the corresponding dissociation constants; without loss of generality $\kappa_1 \leq \cdots \leq \kappa_f$. We assume that there exist $K$ filament classes, i.e., $\kappa_i \in \{\varkappa_1,\ldots,\varkappa_K\}$ for some $\varkappa_1 < \cdots <\varkappa_K$. Denote by $f_i$ the number of filaments with the $i$th class: $f_i := \{\#n: \,\kappa_n = \varkappa_i\}$; $f_1$ represents the number of filaments with the smallest dissociation constant and $f = \sum_{i=1}^K f_i$. For notational convenience, we let $\rho_i :=\kappa_1/\kappa_i$ and $\varrho_i :=\kappa_1/\varkappa_i = \varkappa_1/\varkappa_i$. Throughout the paper, we allow $\varkappa_i$'s to vary with $m$, subject to the constraint that $\rho_i$'s and $\varrho_i$'s are fixed. 
While the preceding assumption on finitely many classes suffices for our application, 
one could consider imposing different regularity conditions on $\kappa$'s, e.g., 
assuming that $\kappa$'s are defined by a continuous profile.
The same system can be interpreted as a closed queueing system: the monomers correspond to customers, the filaments to infinite-buffer single-server queues, and the free pool to an infinite-server queue. Service times in the infinite-server queue are exponentially distributed with rate $\lambda:=\sum \lambda_i$; once the service is completed, the customer joins the $i$th single-server queue with probability $\lambda_i/\lambda$. The number of customer in the $i$th queue is the length of the $i$th filament. After a customer receives an exponential amount of service time with rate $\mu_i$, it leaves the $i$th queue and re-joins the infinite-server queue. The utilization of the $i$th single-server queue is inversely proportional to $\kappa_i$, i.e., the filaments with the smallest dissociation rate correspond to the queues with the highest utilization (most congested).

Let $\pi(l_{1},\ldots,l_{f})$ be the steady-state probability that lengths of filaments $1,\ldots,f $ are equal to $l_{1},\ldots,l_{f}$, respectively. Then, it is well known that $\pi$ is product-form and satisfies (e.g., see~\cite{Kog95})
\begin{equation}
\pi(l_1,\ldots,l_f) = c_{m,f}^{-1} \, m! \frac{ \kappa_1^{-l_1} \cdots \kappa_f^{-l_f}}{(m-l_1-\cdots-l_f)!}, \label{eq:prodpi}
\end{equation}
where $c_{m,f}$ is a normalization constant and $l_1+\cdots+l_f \leq m$. Let $(L_1,\ldots,L_f)$ be a collection of random variables (r.v.s) distributed according to $\pi$; set $M :=m-L_1-\cdots-L_f$. That is, in steady state, $L_i$ is the (random) length of the $i$th filament, and $M$ is the (random) number of monomers in the free pool. The form of the steady-state distribution~\eqref{eq:prodpi} holds for the described closed queueing system under arbitrary probabilistic customer routing~\cite{Kog95}; in that case, the constants $\kappa_i$ should be obtained by solving a system of linear flow-conservation equations.

Next, we introduce independent random variables utilized in the probabilistic representation (see Appendix~\ref{sec:pr} for details). The probabilistic representation~\cite{MLH17} does not require that there exists a finite number of filament classes -- it holds for arbitrary dissociation constants. Define $\sM$ to be $\kappa_1$-mean Poisson random variable and $\sL_i$, $i>f_1$, to be geometric random variables defined by
\[
\Pr[\sL_i \geq l] := \rho_i^l, \quad l=0,1,\ldots;
\]
all defined random variables are independent. It follows that $\expect \sL_i = \rho_i/(1 - \rho_i)$ and $\text{Var}(\sL_i) = \rho_i/(1 - \rho_i)^2 = (\expect \sL_i)^2 /\rho_i$. For notational convenience, we let $\sS := \sum_{i > f_1} \sL_i$, $\sS_i:=\sum_{n: \, \kappa_n = \varkappa_i} \sL_n$ and $\sS_{-i} : = \sS - \sL_i$; $\sS \equiv 0$ if $f_1 = f$. Whenever we write $m^{-1} \expect[\sM+\sS] \to \fkappa_1 + \fss$, both $m^{-1} \expect\sM \to \fkappa_1$ and $m^{-1} \expect\sS \to \fss$ are implied. 

Using the preceding notation and product-form solution \eqref{eq:prodpi}, 
we derive the probabilistic representation for the marginal distribution of individual filament lengths and free monomer pool in Appendix~\ref{sec:pr} -- see equations \eqref{eq:Mpr}, \eqref{eq:L1prf1}, \eqref{eq:L1pr} and \eqref{eq:Lfpr}. 
For example, for $1 \leq i \leq f_1$ (large filaments) and $f_1 >1$, we restate equation \eqref{eq:L1pr} here
\begin{align}
\Pr[L_i = l_i]  &= (f_1 -1) \frac{  \expect [ (m-l_i+f_1 -2 - \sM -  \sS)_{f_1-2} \ind{\sM + \sS \leq m-l_i} ] }{\expect [ (m+f_1 -1 - \sM -  \sS)_{f_1-1} \ind{\sM + \sS \leq m} ]   }; \label{eq:L1pr0}
\end{align}
note that this formula holds for multiple monomer pools (infinite server queues), in which case
$ \sM$ is equal to a sum of independent Poisson variables.
As discussed in introduction, we evaluate the preceding and similar expressions by determining 
the values of $(\sM + \sS)$ around which the expectations in \eqref{eq:L1pr0} are concentrated. 
For the latter, we utilize large deviation properties of $(\sM + \sS)$, which will become apparent from the 
forthcoming analysis.


%
%
%
%
\section{Preliminary large-deviations results}
\label{sec:ld}

Parts of our analysis are based on the large deviations theory. For this reason, we introduce functions that are used throughout the paper. 
In particular, for a sequence of non-negative r.v.s $\{ X(m)\}_{m=1}^\infty$, we define 
the limiting logarithmic moment generating function, $\Lambda_X: [0,\infty) \to \real$, as
\begin{equation}
\Lambda_X(\theta) := \lim_{m\to\infty} \frac1m \log \expect e^{-\theta X(m)},
\label{eq:LX}
\end{equation}
when the limit exists. 
Next, the Legendre-Fenchel transform of $ \Lambda_X$ is known as the rate (Cram\'er) function 
$\ell_X : [0,\infty) \to [0,\infty)$ and is defined by 
\begin{equation}
\label{eq:Xell}
\ell_X(x) = \sup _{{\theta \ge 0}}\{-\theta x- \Lambda_X (\theta )\}.
\end{equation}
In this paper, $\{X(m)\}$ represents a sequence of either Poisson r.v.s or sums of i.i.d. geometric r.v.s, for which we compute the preceding functions explicitly.
Furthermore, under the assumptions of this paper, Cram\'er theorem implies 
\begin{equation}
 \ell_X(x)= - \lim_{m\to\infty} \frac1m \log \Pr[ m^{-1} X(m) \leq x  ] , \quad x \geq 0.  \label{eq:Cramer}
\end{equation}


In particular, when $X(m)=\sM(m)$ is a sequence of Poisson r.v.s with $\expect \sM(m)/m=\kappa_1/m \to \fkappa_1 \geq 0$, as $m \to \infty$, we evaluate~\eqref{eq:LX} and~\eqref{eq:Xell}, respectively, 
using elementary algebra
\begin{equation}
\label{eq:LambdaM}
\Lambda_\sM(\theta)=-\fkappa_1\left(1-e^{-\theta}\right), \quad \theta \geq 0,
\end{equation}
and
\begin{equation}
\ell_{\sM}(x) = \left( \fkappa_1 - x + x \log \frac{x}{\fkappa_1} \right) \ind{x \leq \fkappa_1}, \quad \fkappa_1> 0,x \geq 0;  \label{eq:ellMx}
\end{equation}
in the case $\fkappa_1 = 0$, one has $\ell_{\sM}(x) =0$, $x \geq 0$.
Similarly, when $X(m)=\sS_i(m)$ with $f_i/m\to\ff_i \geq 0$, as $m\to\infty$, direct computations of~\eqref{eq:LX} 
and~\eqref{eq:Xell} yield
\begin{equation}
\label{eq:LambdaSi}
\Lambda_{\sS_i}(\theta) = \ff_i \log \frac{1-\varrho_i}{1-\varrho_i e^{-\theta}}, \quad \theta \geq 0,
\end{equation}
and
\begin{equation}
\ell_{\sS_i}(x) =  \left( x \log \frac{x}{\varrho_i (\ff_i + x)} + \ff_i \log \frac{\ff_i}{(1-\varrho_i) (\ff_i + x)} \right) \ind{x \leq \ff_i \varrho_i/(1-\varrho_i)}, \quad x \geq 0, \label{eq:ellSx}
\end{equation}
with  the standard convention $0 \log 0 = 0$.

The following lemma summarizes relevant properties of the above introduced functions. Since all of our asymptotic results are with respect to passing the number of monomers to infinity ($m \to \infty$), to simplify the notation, we omit the index $m$ from $\sM(m)$ and $\sS(m)$, and simply write $\sM\equiv \sM(m)$ and $\sS\equiv \sS(m)$.

\begin{lemma} \label{lemma:propell} Suppose $m^{-1} \expect[\sM+\sS] \to \fkappa_1+\fss$, as $m\to\infty$. The rate function $\ell_{\sM+\sS}$ is a continuous convex function, strictly decreasing on $[0,\fkappa_1+\fss]$, and $\ell_{\sM+\sS}(x)=0$, $x \geq \fkappa_1+\fss$. For $x\in (0,\fkappa_1+\fss)$, $\ell_{\sM+\sS}$ satisfies $\ell_{\sM+\sS}(x) = -x \theta_x - \Lambda_\sM(\theta_x) - \Lambda_\sS(\theta_x)$, where $x + \Lambda'_\sM(\theta_x) + \Lambda'_\sS(\theta_x)=0$ and $\ell'_{\sM+\sS}(x)= -\theta_x$.
\end{lemma}

\begin{proof} For $x\in (0,\fkappa_1+\fss)$, standard large-deviations arguments  (Chernoff bound and change of measure) yield
\begin{align}
\ell_{\sM+\sS}(x) &= \sup_{\theta \geq 0}\{ -\theta x - \Lambda_\sM(\theta) - \Lambda_\sS(\theta)\} \label{eq:ellLambda} \\
&= -\theta_x x -  \Lambda_\sM(\theta_x) - \Lambda_\sS(\theta_x), \notag
\end{align}
where $\theta_x$ is a unique optimizer (for a given~$x$) and satisfies the first-order condition $x +  \Lambda'_\sM(\theta_x) +\Lambda'_\sS(\theta_x)=0$; note that $\Lambda_\sM$ and $\Lambda_{\sS} = \sum \Lambda_{\sS_i}$ are strictly convex and have continuous derivatives. Given this, one has
\[
\ell'_{\sM+\sS}(x) = -\theta_x - \theta_x' x - \Lambda'_\sM(\theta_x)\, \theta_x' - \Lambda'_\sS(\theta_x)\, \theta_x' = -\theta_x 
\]
The continuity and convexity, as well as strict monotonicity of $\ell_{\sM+\sS}$, follow from~\eqref{eq:ellLambda}.
\end{proof}

While the introduced rate function describes the left tail of the corresponding distribution, its derivative characterizes the density of the distribution, as formalized by the following local limit result.
The proof relies on the analytical properties of the Poisson and geometric random variables. 

\begin{lemma}[Local limit] \label{lemma:local} Let $n \in \nat$ be such that $n/m \to \fn >0$, as $m\to \infty$. If $m^{-1} \expect[\sM+\sS] \to \fkappa_1 + \fss > \fn$, as $m\to \infty$, then
\[
\lim_{m \to \infty} \frac{\Pr[\sM+\sS=n-1]}{\Pr[\sM+\sS=n]} = e^{\ell'_{\sM+\sS}(\fn)}.
\]
\end{lemma}

\begin{proof} See Appendix~\ref{sec:prooflemmalocal}.
\end{proof}

\section{Main results}
\label{sec:results}

Our main results are stated in two subsections, based on the number of filaments ($f_1$) with the smallest dissociation constant. The first subsection covers the case when $f_1$ is linear in the number of monomers~$m$ ($f_1/m \to \ff_1>0$), while we address the sub-linear regime ($f_1/m \to 0$) in the second subsection.

\subsection{Linear regime}
\label{sec:linear}

The following theorem characterizes the limiting distribution of filament lengths in the linear regime. In particular, the limiting lengths are geometrically distributed with the parameter that depends on the large-deviations rate function of $(\sM+\sS)$. That is, in the many-monomer limit, large deviations of $(\sM+\sS)$ define the typical behavior of filament lengths $L_i$.

\begin{theorem} \label{thm:linear} 
Suppose $m^{-1} f_1 \to \ff_1>0$ and $m^{-1} \expect[\sM+\sS] \to \fkappa_1 + \fss$, as $m \to\infty$. Then, as $m \to\infty$,
\begin{equation}
\Pr[L_i \geq l] \to \Pr[\Linf_i \geq l] :=  \left( \rho_i \frac{1-\psi}{\ff_1 + 1 -\psi} \right)^l, \quad l=0,1,\ldots, \label{eq:Linf}
\end{equation}
and
\begin{equation}
m^{-1} M \toP \fkappa_1  \frac{1-\psi}{\ff_1 + 1 -\psi}, \label{eq:Mtogamma}
\end{equation}
where either $\psi=0$ when $\fkappa_1+\fss=0$, or otherwise $\psi \in (0, (\fkappa_1 + \fss) \wedge 1)$ satisfies
\begin{equation}
\log \frac{1-\psi}{\ff_1 + 1-\psi} = \ell_{\sM+\sS}'(\psi), \label{eq:psicondition}
\end{equation}
or equivalently 
\begin{equation}
\psi + \Lambda'_{\sM+\sS}\left(- \log \frac{1-\psi}{\ff_1 + 1-\psi}\right) =0. \label{eq:psicondition2}
\end{equation}
\end{theorem}

\begin{proof}
See Appendix~\ref{sec:proofthmlinear}.
\end{proof}

\begin{remark}[Monomer conservation] Note that $\psi$ is the limiting expected fraction of monomers not present in the filaments with the smallest dissociation constant, because $\ff_1 \, \expect \Linf_1 = 1 - \psi$. Given that $m^{-1} M \in [0,1]$, \eqref{eq:Mtogamma} implies $m^{-1} \expect M \to \fkappa_1  e^{ \ell_{\sM+\sS}'(\psi)}$, as $m \to \infty$, i.e., $\fkappa_1  e^{ \ell_{\sM+\sS}'(\psi)}$ is the expected fraction of monomers in the free pool;
moreover, this convergence is exponentially fast, which can be observed from the poof. 
Finally, Lemma~\ref{lemma:propell} yields the expected fraction of monomers in filaments with indices higher than~$f_1$, as $m\to\infty$:
\begin{align*}
\frac1m\sum_{i=f_1+1}^f \expect \Linf_i &= \frac1m \sum_{i=f_1+1}^f \frac{\rho_i e^{\ell_{\sM+\sS}'(\psi)}}{1-\rho_i e^{\ell_{\sM+\sS}'(\psi)}} \\\ 
&\to -\Lambda'_\sS(-\ell_{\sM+\sS}'(\psi)) \\
&= \psi + \Lambda'_\sM(-\ell_{\sM+\sS}'(\psi)) \\
&= \psi -  \fkappa_1 e^{\ell_{\sM+\sS}'(\psi)}.
\end{align*}
\end{remark}

\begin{remark}[Asymptotic independence]
By examining details of the probabilistic representation and the proof of Theorem~\ref{thm:linear}, as well as other theorems in this section, one can conclude that any finite collection of filament lengths is asymptotically independent under the conditions of the theorems.
We omit making such statements precisely due to cumbersome additional notation. 
\end{remark}

With additional assumptions on the structure of rates $\kappa_i$, the statement of Theorem~\ref{thm:linear} can be made more explicit. The next two corollaries characterize the system with homogeneous filaments and the system with two filament types. 

\begin{corollary}[Homogeneous filaments] \label{coro:homogen} Suppose $\kappa = \kappa_1=\cdots=\kappa_f$, and $m^{-1} (f, \kappa) \to (\ff, \fkappa)$, as $m \to \infty$, for some $\ff > 0$ and $\fkappa  \geq 0$. Then, as $m\to\infty$,
\begin{equation}
\Pr[L_i \geq l]  \to  \left(\frac{ 1- \psi }{ \bar f + 1- \psi} \right)^l, \quad l=0,1,\ldots, \label{eq:Lhomolinear}
\end{equation}
and $m^{-1} M \toP \psi$, where $\psi$ is the unique solution of $\psi^2 - \psi (1+ \ff + \fkappa)  + \fkappa = 0$ on $(0, \fkappa \wedge1)$ when $\fkappa>0$, and $\psi = 0$ when $\fkappa=0$.
\end{corollary}

\begin{proof} Recalling~\eqref{eq:ellMx}, $\psi$ is a root of $p(x):= x^2 - x(1+\ff+\fkappa) + \fkappa$; both roots are nonnegative real, because the corresponding determinant can be written as $\ff^2 + 2 \ff (\fkappa +1) + (\fkappa - 1)^2 > 0$. Moreover, $p(0) = \fkappa$, $p(1) = -\ff$ and $p(\fkappa) =-\ff \fkappa$ result in a single root on $(0,1)$ and $\psi \in (0,\fkappa \wedge 1)$ when $\fkappa>0$ ($\psi=0$ when $\fkappa=0$). The statement of the corollary follows from Theorem~\ref{thm:linear}.
\end{proof}

\begin{remark}[Homogeneous filaments] Note that $m^{-1} M \toP \psi$, as $m\to \infty$, implies $m^{-1} \expect M \to\psi$, as $m\to\infty$; as a consequence, one has $\expect L_i \to (1-\psi)/\ff$, as $m\to\infty$. When $\fkappa>0$, \eqref{eq:Lhomolinear} can be rewritten as $\Pr[L_i \geq l] \to (\psi/\fkappa)^l$, as $m \to \infty$, due to the quadratic equation for $\psi$. In this special case, the validity of the quadratic equation can be verified via
\[
\kappa \, \Pr[L_i > 0]  =  m - f \expect L_i;
\]
the equality holds because both sides of the equation equal to $\expect M$ (due to the flow conservation and monomer conservation laws). Dividing the preceding equality with $m$ and then letting $m\to\infty$ yields
\[
\fkappa \frac{ 1-\psi}{\ff + 1 - \psi} =  \psi,
\]
and the quadratic equation follows.
\end{remark}

\begin{corollary}[Two filament types] 
\label{coro:twoTypes}
Suppose $0< \kappa_1 = \cdots = \kappa_{f_1} < \kappa_{f_1+1} = \cdots = \kappa_{f} = \kappa_1/\rho_f < \infty$ do not vary with $m$, and $m^{-1}(f_1, f_2) \to (\ff_1, \ff_2)$, as $m \to \infty$, for some $\ff_1+\ff_2> \ff_1>0$. Then, as $m \to \infty$, \eqref{eq:Linf} holds and  $m^{-1} M \toP 0$, where  $\psi\in (0,  \ff_2 \frac{\rho_f}{1-\rho_f} \wedge 1)$ is the unique solution of 
\begin{equation}
\psi^2 - \psi \left(1  + \frac{1}{1-\rho_f} \ff_1 + \frac{\rho_f}{1-\rho_f} \ff_2 \right) + \frac{\rho_f}{1 - \rho_f} \ff_2 = 0. \label{eq:2types2eq}
\end{equation}
\end{corollary}


\begin{proof} Note that~\eqref{eq:2types2eq} is of the same form as the quadratic equation in the statement of Corollary~\ref{coro:homogen}. Thus, there exists a unique relevant root on $(0,  \ff_2 \frac{\rho_f}{1-\rho_f} \wedge 1)$. Equation~\eqref{eq:2types2eq} is implied by~\eqref{eq:psicondition} and~\eqref{eq:ellSx}. The statement of the corollary follows from Theorem~\ref{thm:linear}.
\end{proof}

\subsection{Sub-linear regime} 
\label{sec:sublinear}

In this subsection, we consider an asymptotic regime under which $f_1/m \to 0$, as $m \to \infty$. 
In this regime, the qualitative behavior of the filament lengths depends on the relative number of monomers in the system compared to $\expect[\sM+\sS]$. To this end, 
the following theorem characterizes the case when the number of monomers is relatively small, i.e., 
$m<\expect[\sM+\sS]$. 
As in Theorem~\ref{thm:linear}, the distribution of all filaments is geometric and its parameter 
 is determined by the large-deviations behavior of $(\sM+\sS)$.

\begin{theorem} \label{thm:linear2}
Let $f_1/m \to 0$ and $m^{-1} \expect[\sM+\sS] \to \fkappa_1 + \fss > 1$, as $m \to \infty$. Then, as $m\to\infty$, for $i=1,\dots,f$,
\begin{equation}
\Pr[L_i \geq l] \to \Pr[\Linf_i \geq l] :=  \left(\rho_i\, e^{\ell'_{\sM+\sS}(1)} \right)^{l}, \quad l=0,1,\ldots,  \label{eq:Linf3}
\end{equation}
and 
\[
m^{-1} M \toP \fkappa_1 e^{\ell'_{\sM+\sS}(1)},
\]
where $\ell'_{\sM+\sS}(1)$ satisfies
\begin{equation}
1+ \Lambda'_{\sM+\sS}(-\ell'_{\sM+\sS}(1)) = 0.  \label{eq:ell1condition}
\end{equation}
\end{theorem}

\begin{proof}
See Appendix~\ref{sec:proofthmlinear2}.
\end{proof}

\begin{remark} The preceding result relates to Theorem~\ref{thm:linear} in the small $\ff_1$ regime. Indeed, when $\ff_1  \approx 0$, $1- \psi \approx 0$ as well, implying $\psi \approx 1$. Consequently,
\[
\frac{1-\psi}{\ff_1+1-\psi} \approx e^{ \ell_{\sM+\sS}'(1)},
\]
and~\eqref{eq:Linf} and~\eqref{eq:Linf3} coincide.
\end{remark}

Similarly as in Corollaries~\ref{coro:homogen} and~\ref{coro:twoTypes}, 
the limit in Theorem~\ref{thm:linear2} simplifies further for the homogeneous and two filament cases,
as demonstrated by the corollaries below.

\begin{corollary}[Homogeneous filaments] \label{coro:homogen2} Suppose $\kappa = \kappa_1=\cdots=\kappa_f$, and $m^{-1} (f, \kappa) \to (0, \fkappa)$, $m \to \infty$, for some $\fkappa  >1$. Then $\Pr[L_i \geq l]  \to  \fkappa^{-l}$,  $l=0,1,\ldots$, and $m^{-1} M \toP 1$, as $m\to\infty$.
\end{corollary}

\begin{corollary}[Two filament types] \label{coro:2type2}
Suppose $0< \kappa_1 = \cdots = \kappa_{f_1} < \kappa_{f_1+1} = \cdots = \kappa_{f} = \kappa_1/\rho_f < \infty$ do not vary with $m$, and $m^{-1} (f, f_1) \to (\ff,0)$, as $m \to \infty$, for some $\ff> 0$ such that
\[
\ff \expect \sL_f = \ff \frac{\kappa_1}{\kappa_f-\kappa_1} > 1.
\]
Then, as $m \to \infty$, 
\[
\Pr[L_i \geq l] \to \Pr[\Linf_i \geq l] = \left( \frac{1}{1+\ff} \frac{\rho_i}{\rho_f}\right)^{l}, \quad l=0,1,\ldots,
\]
for $i=1,2,\ldots,f$, and $m^{-1} M \toP 0$.
\end{corollary}

\begin{remark}[Two filament types] Under the conditions of Corollary~\ref{coro:2type2}, the limiting expected fraction of monomers present in the filaments with the higher dissociation constant is unity, because $\expect \Linf_f = 1/\ff$ and $(f-f_1)/m \to \ff$, as $m\to\infty$.
\end{remark}

The second theorem in this subsection addresses the case when the number of monomers is 
relatively large, i.e., $m>\expect[\sM+\sS]$. 
Now, the length of large filaments with the smallest dissociation constant $\kappa_1$ grows with the number of monomers, while all other filaments remain geometrically distributed. 
Hence, their lengths need to be appropriately scaled to obtain a non-degenerate 
distribution in the limit. 

\begin{theorem} \label{thm:ld3}
Suppose $m^{-1} \expect[\sM+\sS] \to \fkappa_1 + \fss <1$ and $1 \ll f_1 \ll m$, as $m\to\infty$. Then, as $m\to\infty$, 
\[
\Pr\left[\frac{f_1}{m - \expect[\sM +\sS]} L_i \geq  x \right] \to e^{-x}, \quad 1 \leq i \leq f_1, \quad x \geq 0,
\]
and $\Pr[L_i=l] \to \Pr[\sL_i = l]$, $i > f_1$, $l=0,1,\ldots$. Moreover, $m^{-1} M \toP \fkappa_1$, as $m\to\infty$.
\end{theorem}

\begin{proof}
See Appendix~\ref{sec:proofthmld3}.
\end{proof}

\begin{remark} 
\label{rm:pl}
Theorem~\ref{thm:ld3} relates~\cite{MLH17} (finite $f_1$)  and Theorem~\ref{thm:linear} ($f_1$ linear in $m$). In particular, the results in~\cite{MLH17} imply, for $1 \leq i \leq f_1=2,3,\ldots$,
\[
\Pr\left[\frac{1}{m - \expect[\sM +\sS]} L_i \geq x  \right] \to (1-x)^{f_1-1 }, \quad x \in [0,1],
\]
as $m\to\infty$. Therefore, for $1 \leq i \leq f_1$ and large but finite $f_1$,
\[
\Pr\left[\frac{f_1}{m - \expect[\sM +\sS]} L_i \geq x  \right] \to \left(1-\frac{x}{f_1} \right)^{f_1-1 } \approx e^{-x},
\]
as $m\to\infty$. On the other hand, when $\ff_1$ is positive but small, Theorem~\ref{thm:linear} renders $\psi \approx m^{-1} \expect[\sM+\sS]$ and 
\[
\Pr\left[\frac{f_1}{m - \expect[\sM+\sS]} L_i \geq x \right] \approx \left(1 - \frac{f_1}{m-\expect[\sM+\sS]} \right)^{x \frac{m-\expect[\sM+\sS]}{f_1} } \approx e^{-x},
\]
for $1 \leq i \leq f_1$ and large~$m$; similarly, for $i > f_1$ and large~$m$,
\[
\Pr[L_i \geq l] \approx \rho_i^l 
= \Pr[\sL_i \geq l].
\]
\end{remark}

\begin{remark}[Boundary] The case $m^{-1} \expect[\sM+\sS] \to \fkappa_1 + \fss = 1$, as $m \to \infty$, is critical in terms of the number of monomers $m$. Due to flow conservation, one has $\kappa_i \Pr[L_i>0] = \expect M$, for all $i=1,\ldots,f$. When the number of monomers is ``sufficiently" large, a filament with the smallest dissociation constant (the queue with the highest utilization) 
is nonempty most of the time,
i.e., $\Pr[L_i>0] \approx 1$, for $1 \leq i \leq f_1$; this, in turn, implies $\expect M \approx \kappa_1 = \expect \sM$. When $\Pr[L_i >0]=1$, $1 \leq i \leq f_1$, such filaments can be replaced with Poisson sources, and the remaining network can be interpreted as open. Consequently, $L_i$, $i> f_1$, is geometrically distributed with $\Pr[L_i>0] = \kappa_1/\kappa_i$, i.e., $L_i$ and $\sL_i$ are equal in distribution. Therefore, $\Pr[L_1>0]=1$ implies that the number of monomers in such a network is at least $\expect[\sM+\sS]$, because the expected number of monomers outside the smallest dissociation constant filaments is $\expect[\sM+\sS]$; equivalently, $f_1 \expect L_1 = m - \expect[\sM+\sS]$ when $\Pr[L_1>0]=1$.
\end{remark}

\begin{remark}[Open network equivalent] Under the conditions of Theorem~\ref{thm:ld3}, the filaments (queues) with the smallest dissociation constant (highest utilization) are non-empty with probability one in the limit. Those overloaded queues can be replaced by independent Poisson sources~\cite{ADW13} (with rate $\mu_1$), and the initial closed network can be transformed into its open equivalent. In such a network, the arrival rate to a non-overloaded queue $i>f_1$ is $\lambda_i \kappa_1$, which results in $\kappa_1/\kappa_i$ utilization and $\Pr[L_i=l] \to \Pr[\sL_i=l]$.
\end{remark}

\section{Extensions and discussion} 
\label{sec:ext}
In this section, we illustrate the applicability of our methodology beyond the baseline growth model by discussing recent studies of vehicle sharing~\cite{BLW20} and service~\cite{KBD20} systems, which were based on analytical and transform techniques, respectively. The discovered 
phenomena in Theorems~\ref{thm:linear}--\ref{thm:ld3} are present in systems with large number ($f_1 \gg 1$) of bottleneck queues~\cite{BLW20}, while the probabilistic representation and corresponding techniques provide novel perspectives 
in general, even when there is only one ($f_1=1$) bottleneck 
queue~\cite{KBD20}.

We start by describing briefly the model~\cite{BLW20} for dimensioning a car rental fleet 
subject to a desired service level. In this model, customers correspond to cars and a closed network consists of rental locations, represented by single-server queues, and 
routes between these locations are modeled by infinite-server queues. Specifically, there are $m$~customers (cars), $f$~single-server infinite-buffer queues (locations) and $f^2$~infinite-server queues (routes); the single-server queues are indexed by $i=1,\ldots,f$, while infinite-server queues are indexed by pairs $(i,j)$, $i,j=1,\ldots,f$. Service times are exponential -- the service rates are $\lambda_i$ for single-server queues, and $\mu_{i,j}$ for infinite-server queues. Upon completing service at queue~$i$, a customer joins queue~$(i,j)$ with probability~$p_{i,j}$, after which the customer joins queue~$j$. Let $L_i$ and $M_{i,j}$ be stationary length of queues $i$ and $(i,j)$, respectively. When the network is balanced (equal utilization across the single-server queues), the stationary distribution satisfies, for $l_i, n_{i,j} \geq 0$ such that $\sum l_i + \sum n_{i,j} = m$,
\[
\Pr[L_1=l_1, \ldots, L_f=l_f, M_{1,1}=n_{1,1}, \ldots, M_{f,f}=n_{f,f}]  = c^{-1}_{m,f} \prod_{i,j \in \{1,\ldots,f\}} \left(\frac{\lambda_i p_{i,j}}{\mu_{i,j}} \right)^{n_{i,j}} \frac{1}{n_{i,j}!},
\]
where $c_{m,f}$ is the normalization constant. Developing a probabilistic representation (as in Appendix~\ref{sec:pr}) leads to, for $l=0,1,\ldots$ and $f>1$,
\[
\Pr[L_i=l] = (f-1) \frac{\expect[(m-l+f-2-\sM)_{f-2} \ind{\sM\leq m-l}] }{\expect[(m+f-1-\sM)_{f-1} \ind{\sM\leq m}]},
\]
where $\sM$ is a Poisson random variable, representing the sum of the number of cars on all routes, with
\[
\expect\sM = \sum_{i,j \in \{1,\ldots,f\}} \frac{\lambda_i p_{ij}}{\mu_{ij}}.
\]
Comparing the above distribution to~\eqref{eq:L1pr}, one concludes that the distribution of single-server queue lengths is equal to the  distribution of queues (filaments) length in the baseline model (Section~\ref{sec:model}) with $f$ homogeneous filaments with dissociation rates $\kappa=\expect\sM$. This system is studied in the large ($m,f$) regime, where one expects to see the effects described in our results. 

In~\cite{BLW20}, the authors consider the problem of determining the minimum number of customers in the network such that the probability of a single-server queue being non-empty (service level) is at least $\alpha$, for some given~$\alpha \in (0,1)$. Based on Corollary~\ref{coro:homogen} (Theorem~\ref{thm:linear}) and Corollary~\ref{coro:homogen2} (Theorem~\ref{thm:linear2}), we approximate the service level for finite~$m$ as follows: $\alpha= \Pr[L_i>0] \approx m  \psi/\expect \sM$, where $\psi$ is a solution to $\psi^2 m - \psi(m + f + \expect \sM) + \expect \sM= 0$. Combining the preceding two equations yields an approximation for the minimum number of customers (in the high-$m$ regime):
\begin{equation}
m \approx  \alpha \expect\sM + \frac{\alpha}{1-\alpha}f, \label{eq:mapprox}
\end{equation}
which is in agreement with the results in~\cite{BLW20}. Furthermore, \eqref{eq:mapprox} is in agreement with Theorem~\ref{thm:ld3} as well when $\alpha$ is close to $1$. Indeed, Theorem~\ref{thm:ld3} implies $\alpha = \Pr[L_i>0] \approx e^{-f/(m - \expect \sM)} \approx 1 - f/(m- \expect \sM)$ and
\[
m \approx \expect \sM+   \frac{1}{1-\alpha}f.
\]


We conclude this section with a probabilistic perspective on the recent results in~\cite{KBD20}, which studied a single ($f_1=1$) bottleneck system in a heavy-traffic regime using a transform approach. To this end, consider system with $m$ customers, $h$ infinite-server queues ($M_1,\dots,M_h$) and $f$ single-server queues ($L_1,\ldots,L_f$), one of which is a bottleneck (with index $1$); both $h$ and $f$ are assumed fixed. Then, the joint distribution of non-bottleneck nodes is given by, for $n_1+\cdots+n_g+l_2+\cdots+l_f \leq m$,
\[
\Pr[M_1=n_1,\ldots,M_h=n_h, L_2=l_2, \ldots, L_f=l_f] = c_m^{-1} \prod_{i=1}^h \frac{\vartheta_i^{n_i}}{n_i!} \prod_{j=2}^f \rho_j^{l_j}, 
\]
where $c_m$ is the normalization constant and $\rho_i < 1$, for $i  \geq 2$ ($L_1$ is the length of the bottleneck queue).
This system was analyzed~\cite{KBD20} in the asymptotic regime where $\vartheta_j \to \infty$ and $\rho_i \to 1$ for some or all indices. While the analysis in~\cite{KBD20} was based on transforms, the probabilistic representation for this system could lead to an analysis using the central limit theorem (CLT). In this regard, to develop the probabilistic representation observe that
\begin{equation}
\Pr[M_1=n_1,\ldots,M_h=n_h, L_2=l_2, \ldots, L_f=l_f] = \frac{\Pr[\sM_1=n_1,\ldots,\sM_h=n_h, \sL_2=l_2, \ldots, \sL_f=l_f]}{\Pr[\sM_1+\cdots+\sM_h + \sL_2 + \cdots +\sL_f \leq m ]}; \label{eq:BKD0}
\end{equation}
here $\sM_i$ is mean-$\vartheta_i$ Poisson and $\sL_j$ is geometric as before; all variables are independent. Next, for marginal density of $L_j$, the summation of the preceding formula yields
\begin{equation}
\label{eq:BKD}
\Pr[L_j = l] = \Pr[\sL_j = l]  \frac{\Pr[\sM +\sS_{-j} \leq m - l]}{\Pr[\sM+\sS \leq m ]},
\end{equation}
where $\sM:=\sum \sM_i$ and we recall $\sS=\sum \sS_i$ and $\sS_{-j}=\sS-\sL_j$;
note that due to the single bottleneck, the probabilistic representation does not include expectations of products but rather just probabilities. Informally, the scaling limit in~\cite{KBD20} is such that $m=\lfloor \expect [\sM + \sS]  + \beta \sigma(\sM+\sS)\rfloor \to \infty$, for a fixed $\beta$, which yields $\Pr[\sM+\sS \leq m ] \to \alpha \in (0,1)$. For a single-server node~$j\geq 2$ with the fixed utilization~$\rho_j \in (0,1)$ (not converging to unity), $\max\{\expect \sL_j,\sigma(\sL_j)\} = o( \sigma(\sM+\sS))$ since $\sigma(\sM+\sS)\to \infty$ and $\max\{\expect \sL_j,\sigma(\sL_j)\}=O(1)$. Therefore, the fraction in~\eqref{eq:BKD} converges to one, and $L_j$ is asymptotically geometric: $\Pr[L_j \ge l]\to \Pr[\sL_j \ge l]=\rho_j^{l}$, $l = 0,1,\ldots$. 
Moreover, for $i$'s and $j$'s such that $\sigma(\sM_i) \not = o(\sigma(\sM+\sS))$ and $\sigma(\sL_j) \not = o(\sigma(\sM+\sS))$, respectively, \eqref{eq:BKD0} can serve as a basis for deriving the result in~\cite{KBD20} without relying on Laplace transform. In particular, the joint distribution of such scaled (and centered) queue lengths converges to a conditional joint distribution of the corresponding normal and exponential random variables due to the weak convergence of $(\sM_i  - \expect \sM_i)/\sigma(\sM_i)$ and $\sL_j/\sigma(\sL_j)$.

In addition to the heavy-traffic regime considered in~\cite{KBD20}, one can also consider scalings where $m^{-1}\expect[\sM+\sS]\to c>1$ or $m^{-1}\expect[\sM+\sS]\to c<1$, as in Theorem~\ref{thm:linear2} and Theorem~\ref{thm:ld3}, respectively. When $m^{-1}\expect[\sM+\sS]\to c<1$, both numerator and denominator in \eqref{eq:BKD} converge to one and $\Pr[L_j \ge l]\to \Pr[\sL_j \ge l]$ still holds. Furthermore, the distribution of the bottleneck queue, $L_1$, converges to a uniform distribution~~\cite{MLH17}; see also Remark~\ref{rm:pl}.

\appendix
\section{Probabilistic representation} 
\label{sec:pr}

Recall the definitions of $\sM$, $\sL_i$, $\sS$ and $\sS_{-i}$ from Section~\ref{sec:model}. The probabilistic representation~\cite{MLH17} holds for arbitrary dissociation constants. The normalization constant $c_{m,f}$ can be expressed as follows:
\begin{align*}
c_{m,f} &= m! \sum_{0 \leq l_1 + \cdots + l_f \leq m} \frac{\kappa_1^{-l_1} \cdots \kappa_f^{-l_f}}{(m-l_1-\cdots-l_f)!}  \\
&= m! \, \sum_{i=0}^m \sum_{l_{f_1 +1} + \cdots l_f = i} \sum_{l=0}^{m-i} \binom{l + f_1 -1}{ f_1 -1} \frac{\kappa_1^{-l} } {(m-l-i)!} \prod_{j=f_1 +1}^f \kappa_j^{-l_j} \\
&= m! \, \kappa_1^{-m} e^{\kappa_1} \sum_{i=0}^m \sum_{l_{f_1 +1} + \cdots l_f = i} \sum_{l=0}^{m-i} \binom{l + f_1 -1}{ f_1 -1} \frac{\kappa_1^{m-l-i} e^{-\kappa_1} } {(m-l-i)!} \prod_{j= f_1 +1}^f  \left(\frac{\kappa_1}{\kappa_j} \right)^{l_j}\\
&=   m! \, \kappa_1^{-m} e^{\kappa_1} \sum_{i=0}^m \sum_{l_{ f_1 +1} + \cdots l_f = i} \sum_{l=0}^{m-i} \binom{m +  f_1 -1 - l - i}{ f_1 -1} \frac{\kappa_1^{l} e^{-\kappa_1}} {l!} \prod_{j= f_1 +1}^f  \left(\frac{\kappa_1}{\kappa_j} \right)^{l_j}, 
\end{align*}
where $i$ is the number of monomers in filaments with dissociation constants higher than $\kappa_1$
and in the last equality, we replaced $(m-l-i)$ by $l$. 
The non-negativity of $\sL_n$'s implies $\{\sS=0\} = \{\sL_n = 0,\, n >f_1\}$, which combined with the previous expression yields 
\begin{align}
c_{m,f} &=  m! \, \frac{\kappa_1^{-m} e^{\kappa_1}}{( f_1-1)!} \prod_{j= f_1 +1}^f  \left(1 - \frac{\kappa_1}{\kappa_j} \right)^{-1} \,  \expect \left[ (m+ f_1 -1 - \sM -  \sS)_{ f_1-1} \ind{\sM + \sS \leq m} \right] \notag \\
&= \frac{\expect [ (m+ f_1 -1 - \sM -  \sS)_{ f_1-1} \ind{\sM + \sS \leq m} ] }{( f_1 -1)! \, \Pr[\sM=m] \, \Pr[\sS=0]}. \label{eq:cmfPr}
\end{align}
The stationary probability of having exactly $n$ monomers in the free pool can be obtained by considering all states with $\{M=n\}$ and~\eqref{eq:cmfPr}:
\begin{align}
\Pr[M=n] &= c_{m,f}^{-1} \frac{m!}{n!} \sum_{\sum _{k=1}^{f} l_k = m-n} \kappa_1^{-l_1} \cdots \kappa_f^{-l_f} \notag \\
&= \frac{(f_1 -1)!  \, e^{-\kappa_1} \, \kappa_1^{n} }{n!\, \expect [ (m+f_1 -1 - \sM -  \sS)_{f_1-1} \ind{\sM + \sS \leq m} ]} \sum_{j=0}^{m-n} \binom{m-n-j + f_1-1}{f_1 -1} \, \Pr[\sS=j]  \notag \\
&= \Pr[\sM=n]  \frac{\expect [ (m+f_1 -1 - n -  \sS)_{f_1-1} \ind{\sS \leq m-n} ]}{ \expect [ (m+f_1 -1 - \sM -  \sS)_{f_1-1} \ind{\sM + \sS \leq m} ]}; \label{eq:Mpr}
\end{align} 
here, the index $j$ represents the number of monomers in filaments with non-minimal dissociation constants out of $(m-n)$ monomers not present in the free pool.

The representation for lengths of larger filaments (with constant $\varkappa_1$) is similar. The system with a single filament with the minimal rate $\varkappa_1 = \kappa_1$ is somewhat special -- in that case, one has, for $i=1=f_1$,
\begin{align}
\Pr[L_1 = l_1 ] &= c_{m,f}^{-1}\,  m! \sum_{0 \leq \sum_{j \not = 1} l_j \leq m - l_1} \frac{\kappa_1^{-l_1} \cdots \kappa_f^{-l_f}}{(m - l_1 - \cdots - l_f)!} \notag \\
&= \frac{m! \, \Pr[\sM=m] \, \Pr[\sS=0] }  {\Pr[\sM + \sS \leq m]} \sum_{n=0}^{m-l_1} \sum_{\sum_{j \not = 1} l_j = n} \frac{\kappa_1^{-l_1} \cdots \kappa_f^{-l_f}}{(m - l_1 - n)!} \notag \\
&=  \frac{1} {\Pr[\sM + \sS \leq m]} \sum_{n=0}^{m-l_1} \sum_{\sum_{j \not = 1} l_j = n} \frac{\kappa_1^{m-l_1-n} e^{-\kappa_1}} {(m - l_1 - n)!} \prod_{k \not =1 } \Pr[\sL_k=l_k] \notag \\
&= \frac{\Pr[\sM + \sS = m - l_1]}{\Pr[\sM + \sS \leq m]}, \label{eq:L1prf1}
\end{align}
where the second equality is due to~\eqref{eq:cmfPr}. When $f_1>1$, the probabilistic representation of filament lengths is based on the following equality:
\begin{align*}
\Pr[L_i = l_i]  & = c_{m,f}^{-1} \, m! \, \kappa_i^{-l_i} \sum_{0 \leq \sum_{j \not =i} l_j  \leq m-l_i} \frac{\prod_{k \not = i } \kappa_k^{-l_k} }{(m-l_i - \sum_{k \not = i} l_k)!} \\
& =  \frac{m! \, \kappa_i^{-l_i}}{(m-l_i)!} \frac{c_{m-l_i,f-1}}{c_{m,f}}. 
\end{align*}
In particular, the preceding and~\eqref{eq:cmfPr} yield, for $1 \leq i \leq f_1$ (large filaments) and $f_1 >1$,
\begin{align}
\Pr[L_i = l_i]  &=  \frac{m! \, \kappa_i^{-l_i}}{(m-l_i)!} \frac{  (f_1 -1)! \, \Pr[\sM=m]  \, \expect [ (m-l_i+f_1 -2 - \sM -  \sS)_{f_1-2} \ind{\sM + \sS \leq m-l_i} ] }{\expect [ (m+f_1 -1 - \sM -  \sS)_{f_1-1} \ind{\sM + \sS \leq m} ] \,  (f_1 -2)! \,\Pr[\sM=m-l_i]   }  \notag \\
&= (f_1 -1) \frac{  \expect [ (m-l_i+f_1 -2 - \sM -  \sS)_{f_1-2} \ind{\sM + \sS \leq m-l_i} ] }{\expect [ (m+f_1 -1 - \sM -  \sS)_{f_1-1} \ind{\sM + \sS \leq m} ]   }. \label{eq:L1pr}
\end{align}
A similar reasoning applies to filaments with indices $i>f_1$ (regardless of the value of $f_1$):
\begin{align}
\Pr[L_i = l_i]   &= \frac{m! \, \kappa_i^{-l_i}}{(m-l_i)!} \frac{ \Pr[\sM=m] \, \Pr[\sS=0] \, \expect [ (m-l_i+f_1 -1 - \sM -  \sS_{-i})_{f_1-1} \ind{\sM + \sS_{-i} \leq m-l_i} ] }{\expect [ (m+f_1 -1 - \sM -  \sS)_{f_1-1} \ind{\sM + \sS \leq m} ] \, \Pr[\sM=m-l_i] \, \Pr[\sS_{-i}=0] } \notag  \\
&= \Pr[\sL_i = l_i]  \frac{\expect [ (m-l_i+f_1 -1 - \sM -  \sS_{-i})_{f_1-1} \ind{\sM + \sS_{-i} \leq m-l_i} ] }{\expect [ (m+f_1 -1 - \sM -  \sS)_{f_1-1} \ind{\sM + \sS \leq m} ] }. \label{eq:Lfpr}
\end{align}


%
%
%
\section{Technical results}
\label{sec:tr}

For $a,b > 0$, introduce a strictly concave and decreasing function $g_{a,b}: [0,a] \to \real$:
\begin{equation}
g_{a,b}(x) : = (a+b-x) \log(a+b-x) - (a-x) \log(a-x), \label{eq:gdef}
\end{equation}
with the standard convention $0 \log 0 = 0$, i.e., $g_{a,b}(a) = b \log b$. Note that
\[
g_{a,b}'(x) = \log \frac{a-x}{a+b-x}<0,
\]
and $g_{a,b}'(x) \to -\infty$, as $x \uparrow a$. 

The lemma below is the main technical result.
Informally, it shows that  expectations in equations \eqref{eq:Mpr}, 
\eqref{eq:L1pr} and \eqref{eq:Lfpr} are concentrated on sub-linear intervals.

\begin{lemma} \label{lemma:prodX}
Suppose
\begin{equation}
m^{-1} (\expect[\sM+\sS],n,k) \to (\fkappa_1+\fss,\fn, \fk), \label{eq:nklinearcond}
\end{equation}
as $m\to\infty$, for some $\fkappa_1+\fss \geq 0$ and $\fn,\fk > 0$. For sufficiently small $\epsilon >0$, 
\begin{equation}
\lim_{m \to \infty} \frac{\expect[ (n + k- \sM-\sS)_k \ind{ m^{-1} (\sM+\sS) \in {\mathcal C}_{\epsilon}}  ] } {\expect[ (n + k -\sM-\sS)_k \ind{\sM+\sS \leq n}  ] } = 1;  \label{eq:lemmaXratio}
\end{equation}
here ${\mathcal C}_{\epsilon} := [(\psi - \epsilon )^+,\, \psi + \epsilon]$, where $\psi =0$ when $\fkappa_1+\fss=0$, and otherwise $\psi  \in  (0,\fn \wedge (\fkappa_1+\fss))$ and satisfies
\[
\log \frac{\fn-\psi}{\fn+\fk-\psi} = \ell_{\sM+\sS}'(\psi).
\]
Moreover, 
\begin{equation}
\lim_{m \to \infty} \frac1m \log\left(m^{-k} e^k  \, \expect[ (n + k-\sM-\sS)_k \ind{\sM+\sS \leq n}  ] \right) = g_{\fn,\fk}(\psi) - \ell_{\sM+\sS}(\psi). \label{eq:XVaradhan}
\end{equation}
\end{lemma}


\begin{remark} 
The assumptions of the lemma imply that the probability mass of $m^{-1} (\sM+\sS)$ is concentrated around $\fkappa_1+\fss$ in the limit, as $m \to\infty$. Deviations from $\fkappa_1+\fss$ have exponentially (in $m$) small probabilities. However, because $(n + k- \sM-\sS)_k$ is exponentially large, $\expect[ (n + k- \sM-\sS)_k \ind{\sM+\sS \leq n} ]$ is determined by values of $m^{-1} (\sM+\sS)$ concentrated around $\psi<\fkappa_1+\fss$, when $\fkappa_1+\fss>0$.
\end{remark}

\begin{proof}
{\em Part I.} We start with showing 
\begin{equation}
\sup_{x \in [0, \fn]} \{ g_{\fn,\fk}(x) - \ell_{\sM+\sS}(x) \} = g_{\fn,\fk}(\psi) - \ell_{\sM+\sS}(\psi), \label{eq:suppsi}
\end{equation}
where $\psi$ is a unique maximizer and defined as in the statement of the lemma. The decreasing nature of $g_{\fn,\fk}(\cdot)$ and $\ell_{\sM+\sS}(x) = 0$, for $x \geq \fkappa_1+\fss$, yield
\[
\sup_{x \in [0, \fn]} \{ g_{\fn,\fk}(x) - \ell_{\sM+\sS}(x) \} = \sup_{x \in [0, \fn \wedge (\fkappa_1+\fss)]} \{ g_{\fn,\fk}(x) - \ell_{\sM+\sS}(x) \}.
\]
When $\fkappa_1+\fss=0$, \eqref{eq:suppsi} holds with $\psi=0$, i.e., the supremum equals $g_{\fn,\fk}(0)$. On the other hand, when $\fkappa_1+\fss>0$, the strict concavity of $g_{\fn,\fk}(\cdot)$ and the convexity of $\ell_{\sM+\sS}(\cdot)$ yield the strict concavity of $(g_{\fn,\fk} - \ell_{\sM+\sS})(\cdot)$ on $[0,\fn \wedge (\fkappa_1+\fss)]$;
recall the properties of  $\ell_{\sM+\sS}(\cdot)$ from Lemma~\ref{lemma:propell}.
Now, $g_{\fn,\fk}'(x) \to - \infty$, as $x \uparrow \fn$, and $\ell_{\sM+\sS}'(x) \to 0$, as $x \to \fkappa_1+\fss$, imply that the unique maximum is not achieved at $\fn \wedge (\fkappa_1+\fss)$ but rather on $[0,\fn \wedge (\fkappa_1+\fss))$. Furthermore, $g_{\fn,\fk}'(x) \to \log\frac{\fn}{\fn+\fk}$, as $x \downarrow 0$, and $\ell_{\sM+\sS}'(x) \to -\infty$, as $x \downarrow 0$, yield that the maximum is achieved on $(0,\fn\wedge (\fkappa_1+\fss))$ and the first-order condition $g_{\fn,\fk}'(\psi) = \ell_{\sM+\sS}'(\psi)$ holds, because $\ell_{\sM+\sS}(\cdot)$ is differentiable.

{\em Part II.} The proof of~\eqref{eq:lemmaXratio} consists of obtaining asymptotic estimates for the following three quantities: 
\begin{align*} 
E_1 &:= \expect[ (n + k- \sM-\sS)_k  \ind{ m^{-1} (\sM+\sS)  <  (\psi -  \epsilon )^+ } ], \\
E_2 &:=  \expect[ (n + k- \sM-\sS)_k  \ind{ m^{-1} (\sM+\sS) \in {\mathcal C}_{\epsilon} } ],\\
E_3 &:= \expect[ (n + k- \sM-\sS)_k \ind{\sM+\sS \leq n} \ind{ m^{-1} (\sM+\sS) >  \psi + \epsilon} ].
\end{align*}

In the first step, we obtain an upper bound on $E_1$. It is sufficient to consider the case $\psi>0$ only, because $E_1=0$ when $\psi=0$. To this end, upper bounding summands by its maximum results in (recall the definition of $g_{n,k}$ from~\eqref{eq:gdef})
\begin{align}
E_1
&\leq \psi m \max_{i < m(\psi -  \epsilon )^+} (n+k-i)_k \, \Pr[\sM+\sS \leq i] \notag \\
&\leq \psi m \max_{i < m(\psi -  \epsilon )^+} (n+k-i)_k \, e^{-g_{n,k}(i)} \, \sup_{0 \leq x \leq (\psi - \epsilon)^+ }  e^{g_{n,k}(xm)} \Pr[m^{-1} (\sM+\sS) \leq x]. \label{eq:E1bound1}
\end{align}
Next, for $\epsilon >0$, Sterling's approximation yields, as $m\to\infty$,
\[
\frac1m \log \left( e^k \max_{i < m (\psi -  \epsilon )^+}  (n+k-i)_k \, e^{-g_{n,k}(i)}  \right) \to 0,
\]
while~\eqref{eq:gdef} and~\eqref{eq:nklinearcond} imply, as $m \to \infty$,
\[
\frac1m \log ( m^{-k} e^{g_{n,k}(xm)} ) \to g_{\fn,\fk}(x).
\]
The preceding two limits, \eqref{eq:Cramer} and~\eqref{eq:E1bound1} provide an asymptotic upper bound on $E_1$, for sufficiently small $\epsilon$
\begin{align}
\limsup_{m \to \infty} \frac1m \log (m^{-k}  e^k E_1)  &\leq \sup_{0 \leq x \leq (\psi - \epsilon)^+}  \{ g_{\fn,\fk}(x)  - \ell_{\sM+\sS}(x) \} \notag \\
&< g_{\fn,\fk}(\psi)  - \ell_{\sM+\sS}(\psi), \label{eq:E1bound2}
\end{align}
where the last inequality follows by \eqref{eq:suppsi}.

In the second step, a lower bound on $E_2$ is derived. In particular, for all $\delta \in (0,\epsilon)$, the following holds, when $\psi>0$:
\begin{align}
E_2 
&\geq  \Pr[m^{-1} (\sM+\sS)  \in {\mathcal C}_{\delta}] \, \min_{i/m \in  {\mathcal C}_{\delta}} (n+k-i)_k  \notag \\
&\geq  \left( \Pr[m^{-1} (\sM+\sS) \leq \psi + \delta] - \Pr[m^{-1} (\sM+\sS) \leq (\psi -  \delta)^+ ] \right) \,  (n+k-\lceil (\psi+\delta) m \rceil )_k \notag \\
&= \left( 1 - \frac{\Pr[m^{-1} (\sM+\sS) \leq (\psi -  \delta)^+ ]}{\Pr[m^{-1} (\sM+\sS) \leq \psi + \delta]} \right) \, \Pr[m^{-1} (\sM+\sS) \leq \psi + \delta] \, (n+k-\lceil (\psi+\delta) m \rceil )_k; \label{eq:E2bound1}
\end{align}
the case $\psi=0$ is easier:
\[
E_2 \geq \Pr[m^{-1} (\sM+\sS) \leq \psi + \delta] \,  (n+k-\lceil (\psi+\delta) m \rceil )_k.
\]
In view of~\eqref{eq:Cramer}, when $\psi>0$, one has, $m\to\infty$, 
\[
\frac{1}{m} \log \frac{\Pr[m^{-1} (\sM+\sS) \leq (\psi - \delta )^+ ] }{\Pr[m^{-1} (\sM+\sS) \leq \psi + \delta] } \to - \ell_{\sM+\sS}((\psi - \delta)^+) + \ell_{\sM+\sS}(\psi + \delta) < 0,
\]
where the inequality follows from $\ell_{\sM+\sS}(\cdot)$ being strictly decreasing on $[0,\fkappa_1+\fss]$ ($\fkappa_1+\fss>0$ in this case because $\psi=0$ otherwise). Therefore, the ratio of two probabilities in~\eqref{eq:E2bound1} vanishes, as $m\to\infty$. Using this fact in~\eqref{eq:E2bound1}, letting $\delta \to 0$, and recalling that $g_{\fn,\fk}(\cdot)$ and $\ell_{\sM+\sS}(\cdot)$ are continuous results in
\begin{equation}
\liminf_{m \to \infty} \frac1m  \log ( m^{-k} e^k E_2) \geq  g_{\fn,\fk}(\psi) - \ell_{\sM+\sS}(\psi). \label{eq:E2bound2}
\end{equation}

The third step focuses on an upper bound for $E_3$. For $l \in \nat$, let 
\[
\xi := l^{-1}  (m^{-1} \expect [\sM+\sS]  \wedge 1 - (\psi +\epsilon) )^+.
\]
Note that, for all $m$ large enough, $\xi=0$ or $\xi>0$ depending on the values of $\fkappa_1+\fss$, $\psi$ and $\epsilon$. Next, the following holds
\begin{align}
E_3 
&\leq \sum_{j=1}^l \sum_{(j-1) \xi  \leq i/m - (\psi +\epsilon)  < j \xi }  (n+k-i)_k \, \Pr[\sM+\sS =i] + \sum_{(\expect [\sM+\sS] - \delta m) \vee (\psi + \epsilon)m  \leq i \leq n } (n+k-i)_k \notag \\
&=: E_{31} + E_{32}, \label{eq:E3bound}
\end{align}
with understanding that $E_{31}=0$ when $\xi=0$, and $E_{32}=0$ when $\fkappa_1+\fss>\fn$ for all sufficiently small~$\delta>0$. To consider $E_{31}$, assume that $\psi+\epsilon<\fkappa_1+\fss$ ($\xi>0$), and note that
\[
E_{31} \leq \xi m \sum_{j=1}^l (n+k- \lfloor (\psi +\epsilon)m + (j-1) \xi m \rfloor)_k \, \Pr[m^{-1} (\sM+\sS) \leq \psi +\epsilon + j \xi ].
\]
Letting $m \to\infty$ in the preceding inequality  and using similar arguments as in \eqref{eq:E1bound2}
yield
\begin{align*}
\limsup_{m \to \infty} \frac1m \log (m^{-k} e^k E_{31})  &\leq \max_{j=1,\ldots,l} \{g_{\fn, \fk}( \psi +\epsilon + (j-1) \xi)  - \ell_{\sM+\sS}(\psi +\epsilon + j \xi) \} \\
&\leq \max_{j=0,\ldots,l-1} \{g_{\fn, \fk}( \psi +\epsilon + j \xi)  - \ell_{\sM+\sS}(\psi +\epsilon + j \xi) \} + \Delta,
\end{align*}
where
\[
\Delta :=  \max_{j=1,\ldots,l} \{\ell_{\sM+\sS}( \psi +\epsilon + (j-1) \xi)  - \ell_{\sM+\sS}(\psi +\epsilon + j \xi) \}.
\]
Then, setting further $l \to \infty \; (\xi \to 0)$ and noting that $\ell_{\sM+\sS}(\cdot)$ is continuous ($\Delta \to 0$) results in
\begin{align}
\limsup_{m \to \infty} \frac1m \log (m^{-k} e^k E_{31}) &\leq \sup_{\psi+\epsilon \leq x \leq \fkappa_1 +\fss} \{g_{\fn,\fk}(x) - \ell_{\sM+\sS}(x) \} \notag \\
&< g_{\fn,\fk}(\psi) - \ell_{\sM+\sS}(\psi). \label{eq:E31bound}
\end{align}
The term $E_{32}$ can be bounded by recalling that $\ell_{\sM+\sS}(x)=0$ for $x \geq \fkappa_1+\fss$:
\begin{align*}
\limsup_{m \to \infty} \frac1m \log (m^{-k} e^k E_{32})  &\leq g_{\fn,\fk}((\fkappa_1+\fss)  \vee (\psi + \epsilon))  \\
&= g_{\fn,\fk}((\fkappa_1+\fss)  \vee (\psi + \epsilon)) - \ell_{\sM+\sS}((\fkappa_1+\fss)  \vee (\psi + \epsilon)) \\
&< g_{\fn,\fk}(\psi) - \ell_{\sM+\sS}(\psi).  
\end{align*}
Finally, combining the preceding limit with~\eqref{eq:E1bound2}, \eqref{eq:E2bound2}, \eqref{eq:E3bound}, \eqref{eq:E31bound} yields~\eqref{eq:lemmaXratio}.

{\em Part III.} The last part of the proof of the lemma focuses on~\eqref{eq:XVaradhan}. In particular, one has
\begin{align*}
\frac1m \log &\left(m^{-k} e^k  \, \expect[ (n + k- \sM-\sS))_k \ind{m^{-1} (\sM+\sS) \in {\mathcal C}_\epsilon}  ] \right) \\
&\leq  \frac1m \log\left(m^{-k} e^k  \, (n + k- m(\psi-\epsilon) )_k \Pr[m^{-1} (\sM+\sS) \in {\mathcal C}_\epsilon]   \right)  \\
&\to g_{\fn,\fk}(\psi - \epsilon) - \ell_{\sM+\sS}(\psi + \epsilon),
\end{align*}
as $m\to\infty$, which, together with~\eqref{eq:lemmaXratio}, after $\epsilon \to 0$, results in
\[
\limsup_{m \to \infty} \frac1m \log\left(m^{-k} e^k  \, \expect[ (n + k- \sM-\sS)_k \ind{m^{-1} (\sM+\sS) \leq n}  ] \right) \leq g_{\fn,\fk}(\psi) - \ell_{\sM+\sS}(\psi).
\]
The proof of the lower bound is similar, except one utilizes $\ell_{\sM+\sS}(\cdot)$ being strictly decreasing on $[0,\fkappa_1+\fss]$ (when $\fkappa_1+\fss >0$)
\begin{align*}
\frac1m \log \Pr[m^{-1} (\sM+\sS) \in {\mathcal C}_\epsilon] &\geq \frac1m \log \left( \Pr[m^{-1} (\sM+\sS) \leq \psi+\epsilon] - \Pr[m^{-1} (\sM+\sS) \leq \psi-\epsilon]\right) \\
&=  \frac1m \log \Pr[m^{-1} (\sM+\sS) \leq \psi+\epsilon]  +  \frac1m \log \left( 1 - \frac{\Pr[m^{-1} (\sM+\sS) \leq \psi-\epsilon]}{\Pr[m^{-1} (\sM+\sS) \leq \psi+\epsilon]}\right) \\
&\to \ell_{\sM+\sS}(\psi+\epsilon)
\end{align*}
since $\psi$ is the unique maximizer. 
This completes the proof of Lemma~\ref{lemma:prodX}.
\end{proof}

\begin{lemma} \label{lemma:density} Let $m^{-1} \expect[\sM+\sS] \to \fkappa_1 + \fss > 0$, as $m\to \infty$. If $n \in \{0,1,\ldots\}$ is such that $n/m \to \fn \in [0,\fkappa_1 + \fss]$, as $m\to \infty$, then
\[
\lim_{m \to\infty} \frac1m \log \Pr[\sM+\sS = n] = -\ell_{\sM+\sS}(\fn).
\]
\end{lemma}

\begin{proof} The upper bound is immediate from $\Pr[\sM+\sS = n] \le \Pr[\sM+\sS \le n]$ and \eqref{eq:Cramer}.

For the lower bound, we proceed by induction in the number of filament classes $k$. To this end, note that  
the lemma holds with $\sM+\sS$ replaced by $\sM$ or $\sS_i$ in its statement by direct computation 
since the probability mass functions of a Poisson variable or sums of i.i.d. geometric variables 
admit closed forms. Next, when the number of filament classes $k\ge 1$, 
let $\sZ=\sM+\sum_{j=1}^{k-1} \sS_j$ and assume that the lemma holds when $\sM+\sS$ is replaced by $\sZ$. Then
\begin{align*}
\Pr[\sZ+\sS_k = n] &\geq \sum_{i=\lfloor m (\fn- \fss_k) \rfloor^+}^{\lceil m\fn \rceil \wedge \lceil m \fsz \rceil} \Pr[\sZ=i] \, \Pr[\sS_k = n-i] \\
&\geq \max_{i=\lfloor m (\fn- \fss_k) \rfloor^+, \ldots, \lceil m\fn \rceil \wedge \lceil m \fsz \rceil } \Pr[\sZ=i] \, \Pr[\sS_k = n-i],
\end{align*}
where $m^{-1}\expect \sZ \to \fsz$, as $m \to \infty$, and
\begin{align*}
\liminf_{m \to \infty} \frac1m \log \Pr[\sZ+\sS_k = n] &\geq - \inf_{ (\fn- \fss_k)^+ \leq x \leq \fn \wedge \fsz} \{\ell_\sZ(x) + \ell_{\sS_k}(\fn-x) \} \\
&= - \inf_{ x \in [0, \fn]} \{\ell_\sZ(x) + \ell_{\sS_k}(\fn-x) \} \\
&= - \ell_{\sZ+\sS_k}(\fn); 
\end{align*}
here, the first equality is due to Lemma~\ref{lemma:propell}. The statement of the lemma follows.
\end{proof}

\begin{lemma} \label{lemma:concentrate} Suppose $m^{-1} \expect [\sM+\sS] \to \fkappa_1+\fss >0$, as $n\to\infty$. Let $\sY = \sM$ or $\sY = \sS_i$, $i=1,\ldots,K$. If $n\in\{0,1,\ldots\}$ is such that $n/m \to \fn \in [0,\fkappa_1+\fss]$, as $n\to\infty$, then, for any $\epsilon>0$,
\[
\limsup_{m \to \infty} \frac1m \log \frac{\Pr[\sM+\sS \leq n, \, |\sX+m\Lambda'_\sX(\theta_{\fn})| > \epsilon m]}{\Pr[\sM+\sS \leq n]} < 0,
\]
where $\sX=\sM+\sS-\sY$ and $\theta_{\fn}$ is the solution to $\fn+\Lambda'_{\sM+\sS}(\theta_{\fn}) =0$.
\end{lemma}

\begin{proof} Note that the lemma holds trivially for $\fn=0$, since $\Lambda'_\sX(\theta)=0$ in that case. 
Next, conditioning on $\sX$ and Lemma~\ref{lemma:density} imply
\[
\lim_{m \to \infty} \frac1m \log \Pr[\sX+\sY \leq n] = - \inf_{x \in [0,\fn]} \{\ell_{\sX}(x) + \ell_{\sY}(\fn-x)  \}.
\]
The first-order optimality condition and Lemma~\ref{lemma:propell} yield that the unique optimizer $x^*$ satisfies $\ell'_\sX(x^\ast) = \ell'_\sY(\fn-x^\ast) =  - \theta^*$; moreover, $\ell_\sX(x^*) = -x^* \theta^*-\Lambda_{\sX}(\theta^*)$ and $\ell_\sY(\fn-x^*) = -(\fn-x^*) \theta^*-\Lambda_{\sY}(\theta^*)$, where $x^*+\Lambda'_\sX(\theta^*)=0$ and $\fn-x^*+\Lambda'_\sY(\theta^*)=0$. Therefore,
\begin{equation}
\lim_{m \to \infty} \frac1m \log \Pr[\sX+\sY \leq n] = -\fn \theta^* - \Lambda_{\sM+\sS}(\theta^*), \label{eq:XYninf}
\end{equation}
where $\theta^* = \theta_{\fn}$ and $\theta_{\fn}$ is as in the statement of the lemma. 

On the other hand, an upper bound on the nominator can be obtained as follows:
\begin{align*}
\Pr[\sX+\sY \leq n, \, |\sX + m\Lambda'_\sX(\theta_{\fn})| > \epsilon m] &\leq \sum_{i \in [0,n]: \, |i +m\Lambda'_\sX(\theta_{\fn})| > \epsilon m } \Pr[\sX \leq i] \, \Pr[\sY \leq n-i]  \\
&\leq n \,  \max_{i \in [0,n]: \, |i +m\Lambda'_\sX(\theta_{\fn})| > \epsilon m } \Pr[\sX \leq i] \, \Pr[\sY \leq n-i],
\end{align*}
and, thus,
\begin{align*}
\limsup_{m \to \infty} \frac1m \log \Pr[\sX+\sY \leq n, \, |\sX + m \Lambda'_\sX(\theta_{\fn})| > \epsilon m] &\leq - \inf_{x \in [0,\fn]: \, |x + \Lambda'_\sX(\theta_{\fn})| > \epsilon} \{\ell_{\sX}(x) + \ell_{\sY}(\fn-x) \} \\
&<  - \ell_{\sX}(x^*) - \ell_{\sY}(\fn-x^*),
\end{align*}
where the last inequality is due to the existence of a unique minimum at $x^*$ that satisfies $x^*+\Lambda'_\sX(\theta^*)=0$. Combining the preceding bound and~\eqref{eq:XYninf} yields the statement of the lemma.
\end{proof}

\begin{lemma} \label{lemma:factorial}
 If $k,i \ll n$ and $ki/n \to x$, as $n \to \infty$, then
\[
\lim_{n\to\infty} \frac{(n + k  + i)_k}{(n + k)_k} = e^{x}.
\]
\end{lemma}

\begin{proof} Sterling's approximation implies, as $m \to\infty$,
\begin{align*}
\frac{(n + k  + i)_k}{(n + k)_k} &= \frac{(n+k+i)! \, n!}{ (n+i)! \,(n+k)!}\\
&\sim \frac{(1 + \frac{k+i}{n})^{n+k+i}}{(1 + \frac{i}{n})^{n+i} (1 + \frac{k}{n})^{n+k}}.
\end{align*}
Using the Taylor expansion for $\log(1+x)$ yields the desired result.
\end{proof}

\begin{lemma} \label{lemma:concentrate2}
For $\epsilon >0$, let ${\mathcal D}_\epsilon= \left\{ i \in \{0,\ldots,m\}:  \, |i - \expect[\sM+\sS]| \leq  \epsilon (m - \expect[\sM+\sS]) \right\}$. If $k \ll m$, $n/m \to 1$ and $m^{-1} \expect[\sM+\sS] \to \fkappa_1 + \fss < 1$, as $m \to \infty$, then, for $\epsilon >0$ small enough,
\[
\lim_{m \to \infty} \frac{\expect[(n+k-\sM-\sS)_k \ind{\sM+\sS \in {\mathcal D}_\epsilon}]}{\expect[(n+k-\sM-\sS)_k \ind{\sM+\sS \leq n}]} = 1.
\]
\end{lemma}

\begin{proof} 
For notational simplicity let $\sZ:=\sM+\sS$ and $\tilde m := m- \expect \sZ$. Monotonicity and Lemma~\ref{lemma:propell} yield
\begin{align}
\limsup_{m \to \infty} \frac{1}{m} \log \left( m^{-k} \, \expect[(m + k - \sZ)_k \, \ind{\sZ < \expect \sZ - \epsilon \tilde m}] \right) &\leq 
\limsup_{m \to \infty} \frac{1}{m} \log \left( m^{-k} m\, (m + k)_k \,\Pr[\sZ \leq \expect \sZ - \epsilon \tilde m] \right) \notag \\
&\leq \lim_{m\to\infty} \frac1m \log \Pr[\sZ \leq \expect\sZ-\epsilon \tilde m] < 0. \label{eq:lc10}
\end{align}
The following two limits can be obtained in a similar way:
\begin{align}
\liminf_{m \to \infty} \frac{1}{m} \log \left( m^{-k} \expect[(m+k-\sZ)_k \ind{\sZ \in {\mathcal D}_\epsilon}] \right) &\geq \liminf_{m \to \infty} \frac{1}{m} \log  \left( m^{-k} \, \Pr[\sZ \in {\mathcal D}_\epsilon] \, \min_{j \in {\mathcal D}_\epsilon} (m+k-j)_k \right) \notag \\
&\geq \liminf_{m \to \infty} \frac{1}{m} \log \left( m^{-k} \, \Pr[\sZ \in {\mathcal D}_\epsilon] \, ((1-\epsilon)\tilde m+k)_k \right) = 0, \label{eq:lc20}
\end{align}
and
\begin{align}
\limsup_{m \to \infty} \frac{1}{m} \log &\left( m^{-k} \expect[(m+k-\sZ)_k \ind{\sZ>  \expect \sZ + \epsilon \tilde m}] \right) \notag\\
&\leq \limsup_{m \to \infty} \frac{1}{m} \log \left( m^{-k} ((1+\epsilon) \tilde m+k)_k \, \Pr[\sZ>  \expect \sZ + \epsilon \tilde m] \right) \notag \\
&\leq \lim_{m \to \infty} \frac{1}{m} \Pr[\sZ> \expect \sZ + \epsilon \tilde m] <0. \label{eq:lc30}
\end{align}
Combining~\eqref{eq:lc10}, \eqref{eq:lc20} and \eqref{eq:lc30} yields the statement of the lemma.
\end{proof}

\section{Proofs}

\subsection{Proof of Lemma~\ref{lemma:local}} \label{sec:prooflemmalocal} Note that $\fkappa_1 + \fss > \fn$ implies two cases: (i) $\fkappa_1 >0$, or (ii) $\fkappa_1 = 0$ and $\ff_i>0$ for at least one~$i$. Let $\sY=\sM$ in the first case, and $\sY=\sS_i$ in the second; set $\sX=\sM+\sS-\sY$. For $\theta_{\fn}$ being a solution to $\fn + \Lambda'_{\sM+\sS}(\theta_{\fn})=0$ and $\epsilon>0$, it follows
\begin{align}
\Pr[\sX+\sY=n-1, \, |\sX+ m \Lambda'_{\sX}(\theta_{\fn})| \leq \epsilon m ] &= \sum_{i: \, |i+m \Lambda'_\sX(\theta_{\fn})|\leq \epsilon m} \Pr[\sX=i, \, \sY=n-1-i ] \notag \\
&= \sum_{i: \, |i + m \Lambda'_\sX(\theta_{\fn}) |\leq \epsilon m} \frac{\Pr[\sY=n-i-1]}{\Pr[\sY=n-i]} \Pr[\sX+ \sY=n, \, \sX=i]. \label{eq:con20}
\end{align}

Now, when $\sY=\sM$ ($\fkappa_1>0$ in that case) and $i$ such that $|i + m \Lambda'_\sX(\theta_{\fn}) | \leq \epsilon m$, one has
\[
\frac{n+m \Lambda'_{\sX}(\theta_{\fn}) -\epsilon m}{\kappa_1} \leq \frac{\Pr[\sY=n-i-1]}{\Pr[\sY=n-i]} \leq \frac{n+ m \Lambda'_{\sX}(\theta_{\fn}) +\epsilon m}{\kappa_1},
\]
which, by   \eqref{eq:LambdaM} and Lemma~\ref{lemma:propell} imply, as $\epsilon \to 0$,
\[
\lim_{m \to \infty} \frac{n+m \Lambda'_{\sX}(\theta_{\fn}) \pm \epsilon m}{\kappa_1} \to \frac{-\Lambda'_\sY(\theta_{\fn}) }{\fkappa_1} = e^{-\theta_{\fn}}=e^{\ell'_{\sM+\sS}(\theta_{\fn})}.
\]
On the other hand, when $\sY=\sS_i$ ($\ff_i>0$ in that case) and $i$ such that $|i + m \Lambda'_\sX(\theta_{\fn}) | \leq \epsilon m$, one has
\[
\frac{1}{\varkappa_i} \frac{n+m \Lambda'_{\sX}(\theta_{\fn}) -\epsilon m}{n+m \Lambda'_{\sX}(\theta_{\fn}) -\epsilon m + f_i -1} \leq \frac{\Pr[\sY=n-i-1]}{\Pr[\sY=n-i]} \leq\frac{1}{\varkappa_i} \frac{n+m \Lambda'_{\sX}(\theta_{\fn}) +\epsilon m}{n+m \Lambda'_{\sX}(\theta_{\fn}) +\epsilon m + f_i -1},
\]
which, by   \eqref{eq:LambdaSi} and Lemma~\ref{lemma:propell} imply, as $\epsilon \to 0$,
\[
\lim_{m \to \infty} \frac{1}{\varrho_i} \frac{n+m \Lambda'_{\sX}(\theta_{\fn}) \pm \epsilon m}{n+m \Lambda'_{\sX}(\theta_{\fn}) \pm \epsilon m + f_i -1} \to \frac{1}{\varrho_i} \frac{-\Lambda'_\sY(\theta_{\fn}) }{-\Lambda'_\sY(\theta_{\fn}) + \ff_i } = e^{-\theta_{\fn}}=e^{\ell'_{\sM+\sS}(\theta_{\fn})}.
\]
Combining the preceding with~\eqref{eq:con20} results in
\begin{equation}
\lim_{m \to \infty} \frac{\Pr[\sX+\sY=n-1, \, |\sX+ m \Lambda'_{\sX}(\theta_{\fn})| \leq \epsilon m ]}{\Pr[\sX+\sY=n, \, |\sX+ m \Lambda'_{\sX}(\theta_{\fn})| \leq \epsilon m ]} \to e^{\ell'_{\sM+\sS}(\theta_{\fn})}. \label{eq:con40}
\end{equation}
In addition, using Lemma~\ref{lemma:density} and Lemma~\ref{lemma:concentrate} (see Appendix~\ref{sec:tr}),
as $m\to \infty$,
\[
1\ge \frac{\Pr[\sX+\sY=n, \, |\sX+ m \Lambda'_{\sX}(\theta_{\fn})| \leq \epsilon m ]}{\Pr[\sX+\sY=n]}
\ge 1- \frac{\Pr[\sX+\sY=n, \, |\sX+ m \Lambda'_{\sX}(\theta_{\fn})| > \epsilon m ]}{\Pr[\sX+\sY=n]}
\to 1.
\]
Finally, applying the preceding observation, the statement of the lemma follows from
\[
\Pr[\sM+\sS=n-1] = \Pr[\sX+\sY=n-1, \, |\sX+m\Lambda'_\sX(\theta_{\fn})| \leq \epsilon m ] + \Pr[\sX+\sY=n-1, \, |\sX+m\Lambda'_\sX(\theta_{\fn})| > \epsilon m ],
\]
\eqref{eq:con40}, Lemma~\ref{lemma:density} and Lemma~\ref{lemma:concentrate}.

\subsection{Proof of Theorem~\ref{thm:linear}} \label{sec:proofthmlinear} The equivalence of~\eqref{eq:psicondition} and~\eqref{eq:psicondition2} follows from Lemma~\ref{lemma:propell}.

Regarding~\eqref{eq:Linf}, probabilistic representation~\eqref{eq:L1pr} and Lemma~\ref{lemma:prodX} imply, for $1 \leq i \leq f_1$, $l=0,1,\ldots$, and $\epsilon >0 $ small enough,
\begin{align*}
\Pr[L_i = l] &\sim (f_1 -1) \frac{  \expect [ (m - l +f_1 -2 - \sM -  \sS)_{f_1-2}  \ind{|\sM + \sS - \psi m| \leq \epsilon m} ] }{\expect [ (m+f_1 -1 - \sM -  \sS)_{f_1-1}  \ind{|\sM + \sS - \psi m| \leq \epsilon m} ]   }  \\
&= (f_1 -1) \frac{  \expect [ (m - l +f_1 -2 - \sM -  \sS)_{f_1-l-2} (m - \sM - \sS)_l \ind{|\sM + \sS - \psi m| \leq \epsilon m} ] }{\expect [ (m+f_1 -1 - \sM -  \sS)_{l+1} (m-l+f_1 -2 - \sM -  \sS)_{f_1-l-2} \ind{|\sM + \sS - \psi m| \leq \epsilon m} ]   }  \\
& \leq (f_1 -1) \frac{(m-(1-\epsilon) \psi m)_l}{(m+f_1-1-(1+\epsilon) \psi)_{l+1}} \frac{  \expect [ (m - l +f_1 -2 - \sM -  \sS)_{f_1-l-2}  \ind{|\sM + \sS - \psi m| \leq \epsilon m} ] }{\expect [ (m-l+f_1 -2 - \sM -  \sS)_{f_1-l-2} \ind{|\sM + \sS - \psi m| \leq \epsilon m}  ]   } ,
\end{align*}
as $m \to \infty$. Letting first $m\to\infty$ and then $\epsilon \to 0$, results in an upper bound:
\begin{align*}
\limsup_{m \to \infty} \Pr[L_i = l] &\leq \frac{\ff_1}{\ff_1 +1 - (\psi + \epsilon) } \left( \frac{1-(\psi - \epsilon) }{ \ff_1 + 1 - (\psi + \epsilon)  }\right)^l \\
&\to  \frac{\ff_1}{ \ff_1 +1 - \psi } \left( \frac{1- \psi  }{ \ff_1 + 1 - \psi   }\right)^l. 
\end{align*}
The lower bound follows exactly the same steps, except the signs in front of $\epsilon$ are reversed. This proves~\eqref{eq:Linf} for the filaments with the smallest dissociation constant.

For $i > f_1$, probabilistic representation~\eqref{eq:Lfpr}, $\ell_{\sM+\sS_{-i}}(x) = \ell_{\sM+\sS}(x)$, $x \geq 0$, and the preceding argument yield, as $m\to\infty$,
\begin{align*}
\frac{\Pr[L_i = l] }{\Pr[L_i = 0] } &= \frac{\Pr[\sL_i = l]  \, \expect [ (m-l+f_1 -1 - \sM -  \sS_{-i})_{f_1-1} \ind{\sM + \sS_{-i} \leq m-l} ] }{\Pr[\sL_i = 0]  \, \expect [ (m + f_1 -1 - \sM -  \sS_{-i})_{f_1-1} \ind{\sM + \sS_{-i} \leq m} ] } \\
&\to \left( \frac{\kappa_1}{\kappa_i} \frac{1-\psi}{\ff_1 + 1 -\psi} \right)^{l}, \quad l=0,1,2,\ldots
\end{align*}

The rest of the proof justifies~\eqref{eq:Mtogamma}, which is somewhat more involved since~$\expect M$ can be increasing. The starting point is~\eqref{eq:Mpr}:
\begin{equation}
\frac{\Pr[M=n] }{\Pr[M=j]} = \frac{\Pr[\sM=n] \,\expect [ (m+f_1 -1 - n -  \sS)_{f_1-1} \ind{\sS \leq m-n} ]}{ \Pr[\sM=j] \,\expect [ (m+f_1 -1 - j -  \sS)_{f_1-1} \ind{\sS \leq m-j} ]}, \label{eq:Mratio103}
\end{equation}
where $n,j \leq m$ are non-negative integers. For notational simplicity let
\[
\gamma:= \fkappa_1 \frac{1-\psi}{\ff_1+1-\psi} \leq \psi,
\] 
where the inequality follows from~\eqref{eq:psicondition}, \eqref{eq:LX} and Lemma~\ref{lemma:propell}:
\[
\psi - \fkappa_1 e^{\ell'_{\sM+\sS}(\psi)} = \psi + \Lambda_\sM(-\ell'_{\sM+\sS}(\psi)) = - \Lambda_\sS(-\ell'_{\sM+\sS}(\psi)) \geq 0.
\]
Furthermore, for $x \in [0, \fkappa_1+\fss]$, $y \in [0,x]$,
definition \eqref{eq:Xell} and triangular inequality yield
\begin{equation}
\nonumber
\ell_{\sM +\sS}(x) \le \ell_{\sM}(y) + \ell_{\sS}(x-y),  
\end{equation}
where, in conjunction with Lemma~\ref{lemma:propell}, the equality is achieved for
$\ell'_{\sM}(y^*)= \ell'_{\sS}(x-y^*) =\ell'_{\sM +\sS}(x)$, i.e., 
\begin{equation}
\ell_{\sM +\sS}(x) = \inf_{y \in [0,x]} \{\ell_{\sM}(y) + \ell_{\sS}(x-y) \}
=\ell_{\sM}(y^*) + \ell_{\sS}(x-y^*),  \label{eq:jointell}
\end{equation}
for $x \in [0, \fkappa_1+\fss]$. In particular, when $x=\psi$, 
$\ell'_{\sM}(y^*)=\ell'_{\sM +\sS}(\psi)$, \eqref{eq:psicondition}, \eqref{eq:LambdaM} and Lemma~\ref{lemma:propell}, 
yield $y^*=\gamma$, as defined earlier, resulting in
\begin{equation}
\ell_{\sM+\sS}(\psi) = \ell_\sM(\gamma) + \ell_\sS(\psi-\gamma).  \label{eq:ellsum}
\end{equation}

To argue convergence in probability, we consider, for sufficiently small $\epsilon>0$,
\begin{align}
\Pr[ |M - \gamma m | > \epsilon m] &= \sum_{|n - \gamma m| > \epsilon m} \Pr[M=n] \notag \\
&\leq  \sum_{|n - \gamma m| > \epsilon m} \frac{\Pr[M=n]}{\Pr[M=\lfloor \gamma m\rfloor]} \notag \\
&\leq m \,  \max_{|n - \gamma m| > \epsilon m} \frac{\Pr[M=n]}{\Pr[M=\lfloor \gamma m\rfloor]}. \label{eq:Mgammabound}
\end{align}
The goal is to show that the ratio in the last expression vanishes when $n$ is not in the neighborhood of $\gamma m$. In view of the numerator in~\eqref{eq:Mratio103}, note that Lemma~\ref{lemma:prodX} yields, for $n$ such that $n/m \to \alpha \in [0,1]$,
\begin{align}
\frac1m  \log &\left( m^{-f_1} e^{f_1} \, \Pr[\sM = n] \, \expect [ (m+f_1 -1 -n -  \sS)_{f_1-1} \ind{\sS \leq m- n} ]  \right)   \notag \\
&\to \sup_{x\in [0,1-\alpha]} \{ g_{1-\alpha,\ff_1}(x)  - \ell_{\sS}(x) \} - \ell_{\sM}(\alpha) \notag \\
&= \sup_{x\in  [0,1-\alpha]} \{ g_{1,\ff_1}(\alpha+x) -  \ell_{\sM}(\alpha) - \ell_{\sS}(x) \} \notag \\
&=: \sup_{x\in  [0,1-\alpha]}   G_{1,\ff_1}(\alpha,x). \label{eq:Galphabound}
\end{align}
Next, we upper bound the preceding supremum for two cases: $\alpha\in [0, \gamma-\epsilon]$ and $\alpha \in [\gamma+\epsilon,1]$. For the first case (when $\psi>0$, otherwise $\gamma=0$), \eqref{eq:jointell}, \eqref{eq:ellsum}, and the properties of $\ell_\sM$ and $\ell_\sS$ result in $\ell_{\sM+\sS}(\alpha+x) < \ell_\sM(\alpha) + \ell_\sS(x)$, for $\alpha+x > \psi-\delta$ and sufficiently small~$\delta$. This, together with \eqref{eq:jointell}, yields
\begin{align}
\sup_{\alpha \in [0,\gamma-\epsilon]} \sup_{x\in  [0,1-\alpha]}  G_{1,\ff_1}(\alpha,x) &\leq \sup_{\alpha \in [0,\gamma-\epsilon]} \left\{ \left( \sup_{x\in  [0,\psi-\delta-\alpha]}  G_{1,\ff_1}(\alpha,x) \right) \vee \left( \sup_{x\in  [\psi-\delta-\alpha,1-\alpha]}  G_{1,\ff_1}(\alpha,x) \right) \right\} \notag \\
&\leq \left( \sup_{\alpha + x \in  [0,\psi-\delta]}  G_{1,\ff_1}(\alpha,x) \right) \vee \left(\sup_{\alpha \in [0,\gamma-\epsilon]} \sup_{\alpha + x\in  [\psi-\delta,1]}  G_{1,\ff_1}(\alpha,x) \right) \notag \\
&< g_{1,\ff_1}(\psi) -  \ell_{\sM+\sS}(\psi). \label{eq:Galphabound1}
\end{align}
Similarly, for the second case, $\ell_{\sM+\sS}(\alpha+x) < \ell_\sM(\alpha) + \ell_\sS(x)$, for $\alpha+x < \psi+\delta$ and sufficiently small~$\delta$, and~\eqref{eq:jointell} result in 
\begin{align}
\sup_{\alpha \in [\gamma+\epsilon,1]} \sup_{x\in  [0,1-\alpha]}  G_{1,\ff_1}(\alpha,x) &\leq \sup_{\alpha \in [\gamma+\epsilon,1]} \left\{ \left( \sup_{x\in  [0,\psi+\delta-\alpha]}  G_{1,\ff_1}(\alpha,x) \right) \vee \left( \sup_{x\in  [\psi+\delta-\alpha,1-\alpha]}  G_{1,\ff_1}(\alpha,x) \right) \right\} \notag \\
&\leq  \left( \sup_{\alpha \in [\gamma+\epsilon,1]} \sup_{x\in  [0,\psi+\delta-\alpha]}  G_{1,\ff_1}(\alpha,x) \right) \vee \left( \sup_{\alpha+x\in  [\psi+\delta,1]}  G_{1,\ff_1}(\alpha,x) \right) \notag \\
&< g_{1,\ff_1}(\psi) -  \ell_{\sM+\sS}(\psi). \label{eq:Galphabound2}
\end{align}
Finally, regarding the denominator in~\eqref{eq:Mratio103} for $j=\lfloor \gamma m\rfloor$, $\psi< \fkappa_1+\fss$ and~\eqref{eq:jointell} imply $\gamma \leq \fkappa_1$, and it follows that, as $m\to\infty$,
\begin{align*}
\frac1m  \log &\left( m^{-f_1} e^{f_1} \, \Pr[\sM =\lfloor \gamma m \rfloor ] \, \expect [ (m+f_1 -1 - \lfloor \gamma m \rfloor -  \sS)_{f_1-1} \ind{\sS \leq m- \lfloor \gamma m \rfloor} ]  \right)   \\
&\to \sup_{x\in [0,1-\gamma]} \{ g_{1-\gamma,\ff_1}(x)  - \ell_{\sS}(x) \} - \ell_{\sM}(\gamma) \\
&\geq  g_{1,\ff_1}(\psi) -  \ell_{\sM}(\gamma) - \ell_{\sS}(\psi-\gamma)  \\
&=  g_{1,\ff_1}(\psi) -  \ell_{\sM+\sS}(\psi).
\end{align*}
Combining the preceding with \eqref{eq:Mratio103}, \eqref{eq:Galphabound}, \eqref{eq:Galphabound1}, \eqref{eq:Galphabound2}, and \eqref{eq:Mgammabound} yields~\eqref{eq:Mtogamma}.

\subsection{Proof of Theorem~\ref{thm:linear2}} \label{sec:proofthmlinear2} The first step is to argue, for $\epsilon >0$,
\begin{equation}
\lim_{m\to\infty} \frac{\expect [ (m-l+f_1 -2 - \sM -  \sS)_{f_1-2} \ind{\sM + \sS \leq (1-\epsilon)m} ]}{\expect [ (m-l+f_1 -2 - \sM -  \sS)_{f_1-2} \ind{\sM + \sS \leq m-l} ]} = 0. \label{eq:thm3lim10}
\end{equation}
To this end, an upper bound on on the nominator is as follows:
\begin{align*}
\expect [ (m-l+f_1 -2 - \sM -  \sS)_{f_1-2} \ind{\sM + \sS \leq (1-\epsilon)m} ] &\leq   m \max_{i=0,\ldots,(1-\epsilon)m} (m-l+f_1 -2 - i)_{f_1-2} \Pr[\sM+\sS \leq i] \\
&\leq m \, (m-l+f_1 -2)_{f_1-2} \Pr[\sM+\sS \leq (1-\epsilon)m ],
\end{align*}
and, hence,
\begin{equation}
\limsup_{m\to\infty} \frac1m \log \left(m^{-f_1} \expect [ (m-l+f_1 -2 - \sM -  \sS)_{f_1-2} \ind{\sM + \sS \leq (1-\epsilon)m} ] \right) \leq -\ell_{\sM+\sS}(1-\epsilon). 
\label{eq:thm3ub10}
\end{equation}
On the other hand, one has
\[
\expect [ (m-l+f_1 -2 - \sM -  \sS)_{f_1-2} \ind{\sM + \sS \leq m-l} ] \geq (\epsilon m -l + f_1-2)_{f_1-2} \, \Pr[(1-\epsilon) m < \sM+\sS \leq m-l],
\]
which together with Lemma~\ref{lemma:propell} implies
\[
\liminf_{m\to\infty} \frac1m \log \left(m^{-f_1} \expect [ (m-l+f_1 -2 - \sM -  \sS)_{f_1-2} \ind{\sM + \sS \leq m-l} ] \right) \geq -\ell_{\sM+\sS}(1). 
\]
The preceding limit, \eqref{eq:thm3ub10} and Lemma~\ref{lemma:propell} yield~\eqref{eq:thm3lim10}.

We consider filaments with the smallest dissociation constant first. The case $f_1=1$ is straightforward -- \eqref{eq:Linf3} for $i=1$ follows from~\eqref{eq:L1prf1}
 and Lemma~\ref{lemma:local}, since $\Pr[L_1=l]/\Pr[L_1=0]=\Pr[\sM+\sS=m-l]/\Pr[\sM+\sS=m]$. When $f_1>1$, for $1 \leq i \leq f_1$, representation~\eqref{eq:L1pr} yields, for $l=0,1,\ldots$,
\begin{align*}
\frac{\Pr[L_i = l] }{\Pr[L_i = 0] } &= \frac{\expect [ (m-l+f_1 -2 - \sM -  \sS)_{f_1-2} \ind{\sM + \sS \leq m-l} ] }{\expect [ (m + f_1 -2 - \sM -  \sS)_{f_1-2} \ind{\sM + \sS \leq m} ] } \\
&= \frac{\sum_{n=0}^{m-l} (n+f_1 -2)_{f_1-2} \Pr[\sM + \sS = m-l -n] }{\sum_{n=0}^{m} (n+f_1 -2)_{f_1-2} \Pr[\sM + \sS = m -n]  } \\
&\sim  \frac{\sum_{n=0}^{\lfloor \epsilon m \rfloor} (n+f_1 -2)_{f_1-2} \Pr[\sM + \sS = m-l -n] }{\sum_{n=0}^{\lfloor \epsilon m \rfloor} (n+f_1 -2)_{f_1-2} \Pr[\sM + \sS = m-n]  },
\end{align*}
as $m \to \infty$, where the last limit is due to~\eqref{eq:thm3lim10}. Further bounding results in 
\begin{align*}
\min_{m - \lfloor \epsilon m \rfloor \leq n \leq m} &\frac{\Pr[\sM + \sS = n-l] }{\Pr[\sM + \sS = n]  } \\
& \leq \frac{\sum_{n=0}^{\lfloor \epsilon m \rfloor} (n+f_1 -2)_{f_1-2} \Pr[\sM + \sS = m-l -n] }{\sum_{n=0}^{\lfloor \epsilon m \rfloor} (n+f_1 -2)_{f_1-2} \Pr[\sM + \sS = m-n]  } \\
& \hspace{1in}\leq \max_{m - \lfloor \epsilon m \rfloor \leq n \leq m} \frac{\Pr[\sM + \sS = n-l] }{\Pr[\sM + \sS = n]  },
\end{align*}
which after letting $m\to\infty$ and applying Lemma~\ref{lemma:local} yields
\[
\inf_{1-\epsilon \leq x \leq 1} e^{l \ell'_{\sM+\sS}(x)} \leq \liminf_{m\to\infty} \frac{\Pr[L_i = l] }{\Pr[L_i = 0] } \leq \limsup_{m\to\infty} \frac{\Pr[L_i = l] }{\Pr[L_i = 0] } \leq \sup_{1-\epsilon \leq x \leq 1} e^{l \ell'_{\sM+\sS}(x)}.
\]
Letting $\epsilon \to 0$ in the preceding expression, together with Lemma~\ref{lemma:propell} results in~\eqref{eq:Linf3}, for $1 \leq i \leq f_1$. For $i>f_1$, \eqref{eq:Lfpr} renders
\[
\frac{\Pr[L_i = l] }{\Pr[L_i = 0] } = \frac{\Pr[\sL_i = l] }{\Pr[\sL_i = 0] } \frac{\expect [ (m-l+f_1 -1 - \sM -  \sS_{-i})_{f_1-1} \ind{\sM + \sS_{-i} \leq m-l} ] }{\expect [ (m + f_1 -1 - \sM -  \sS_{-i})_{f_1-1} \ind{\sM + \sS_{-i} \leq m} ] },
\]
$l=0,1,\ldots$. We omit further details, because the proof follows exactly the same steps as in the case of $1 \leq i \leq f_1$.

Next, we consider the limit for $M$. Probabilistic representation~\eqref{eq:Mpr} yields
\begin{equation*}
\frac{\Pr[M=n] }{\Pr[M=j]} = \frac{\Pr[\sM=n] \,\expect [ (m+f_1 -1 - n -  \sS)_{f_1-1} \ind{\sS \leq m-n} ]}{ \Pr[\sM=j] \,\expect [ (m+f_1 -1 - j -  \sS)_{f_1-1} \ind{\sS \leq m-j} ]}, \label{eq:Mratio1033}
\end{equation*}
for non-negative integers $n,j \leq m$. As in the proof of Theorem~\ref{thm:linear}, let $\gamma:=\fkappa_1 e^{\ell'_{\sM+\sS}(1)}$, and recall from the proof of Theorem~\ref{thm:linear} that Lemma~\ref{lemma:propell} imples
\begin{align*}
\ell_{\sM+\sS}(1) &= \inf_{x \in [0,1]} \{ \ell_\sM(x) + \ell_\sS(1-x) \} \\
&=  \ell_\sM(\gamma) + \ell_\sS(1-\gamma).
\end{align*}
Note that, for $\fkappa_1>0$ and $n$ such that $n/m \to \alpha \geq 0$, as $m \to \infty$,
\[
\frac{1}{m} \log \left( m^{-f_1} \Pr[\sM=n] \,\expect [ (m+f_1 -1 - n -  \sS)_{f_1-1} \ind{\sS \leq m-n} ] \right) \to -\fkappa_1 + \alpha +\alpha \log \frac{\fkappa_1}{\alpha} - \ell_{\sS}(1-\alpha).
\]
Therefore, for $\kappa_1 \geq 0$ and $n/m \to \alpha \geq 0$ such that $|\alpha-\gamma| > \epsilon >0$, the following holds: 
\begin{align*}
\lim_{n \to \infty} \frac1m\log \frac{\Pr[\sM=n]}{\Pr[\sM=\lfloor \gamma m \rfloor]} &\leq - \ell_\sM(\alpha) - \ell_\sS(1-\alpha) + \ell_{\sM}(\gamma) + \ell_\sS(1-\gamma)  \\
&= - \ell_\sM(\alpha) - \ell_\sS(1-\alpha) + \ell_{\sM+\sS}(1) < 0,
\end{align*}
where the last inequality is due to $\gamma$ being a unique optimizer. Hence, the limit for $M$ is implied by the preceding limit and 
\[
\Pr[|m^{-1} M- \gamma| > \epsilon] \leq m \max_{n: \, |n/m-\gamma|>\epsilon} \Pr[M=n].
\]
Finally, \eqref{eq:ell1condition} follows from Lemma~\ref{lemma:propell}, which concludes the proof.

\subsection{Proof of Theorem~\ref{thm:ld3}} \label{sec:proofthmld3}
Consider $L_i$, $1 \leq i \leq f_1$, first. For $\epsilon>0$ and $l$ such that $l f_1 /(m - \expect[\sM+\sS]) \to x$, as $m \to \infty$, Lemma~\ref{lemma:concentrate2} yields
\begin{align*}
\frac{m - \expect[\sM+\sS]}{f_1} \Pr[L_i = l] &= (m - \expect[\sM+\sS]) \frac{f_1-1}{f_1} \frac{\expect[(m-l+f_1-2-\sM-\sS)_{f_1-2} \, \ind{\sM+\sS \leq m-l}]}{\expect[ (m+f_1-1-\sM-\sS)_{f_1-1} \, \ind{\sM+\sS \leq m}]} \\
&\sim (m - \expect[\sM+\sS]) \frac{\expect[(m-l+f_1-2-\sM-\sS)_{f_1-2} \, \ind{\sM+\sS \in {\mathcal D}_\epsilon}]}{\expect[ (m+f_1-1-\sM-\sS)_{f_1-1} \, \ind{\sM+\sS \in {\mathcal D}_\epsilon}]} \\
&\leq \frac{m - \expect[\sM+\sS]}{(1-\epsilon)(m - \expect[\sM+\sS])-l+f_1-1} \max_{n \in {\mathcal D}_\epsilon}  \frac{(m-l+f_1-1- n)_{f_1-1}}{ (m+f_1-1- n)_{f_1-1} } \\
&\to \frac{1}{1-\epsilon} e^{-\frac{x}{1+\epsilon}},
\end{align*}
as $m\to \infty$, where we used $\sum a_i /\sum b_i \leq \max a_i/b_i$ and the last limit follows from  Lemma~\ref{lemma:factorial}. Letting $\epsilon \to 0$ yields an upper bound. A matching lower bound can be obtained using the same reasoning and $\sum a_i /\sum b_i \geq \min a_i/b_i$. Combining the two bounds results in
\[
\frac{m - \expect[\sM+\sS]}{f_1} \Pr[L_i = l] \to e^{-x},
\]
where $l$ and $x$ such that $l f_1 /(m - \expect[\sM+\sS]) \to x$, as $m \to \infty$.

The analysis for $L_i$, $i > f_1$, is similar, expect that $l=0,1,\ldots$ and $lf_1/(m - \expect[\sM+\sS]) \to 0$, as $m\to\infty$. Repeating the steps of the previous part yields, as $m \to \infty$,
\begin{align*}
\frac{\Pr[L_i = l] }{\Pr[L_i = 0] } &= \frac{\Pr[\sL_i = l]}{\Pr[\sL_i = 0] } \frac{\expect [ (m-l+f_1 -1 - \sM -  \sS_{-i})_{f_1-1} \ind{\sM + \sS_{-i} \leq m-l} ] }{\expect [ (m + f_1 -1 - \sM -  \sS_{-i})_{f_1-1} \ind{\sM + \sS_{-i} \leq m} ] } \\
&\to \frac{\Pr[\sL_i = l]}{\Pr[\sL_i = 0] }. 
\end{align*}

Next, we consider the convergence of $m^{-1}M$. To this end, for $\epsilon>0$, one has
\begin{align}
\Pr[|M-\fkappa_1 m| > \epsilon m] &\leq \sum_{0 \leq n < (\fkappa_1 - \epsilon)m} \frac{\Pr[M = n]}{\Pr[M=\lfloor \fkappa_1 m\rfloor]} + \sum_{(\fkappa_1 + \epsilon)m < n \leq m} \frac{\Pr[M = n]}{\Pr[M=\lfloor \fkappa_1 m\rfloor]} \notag \\
&\leq m \max_{0 \leq n < (\fkappa_1 - \epsilon)m} \frac{\Pr[M = n]}{\Pr[M=\lfloor \fkappa_1 m\rfloor]} + m \max_{(\fkappa_1 + \epsilon)m < n \leq m} \frac{\Pr[M = n]}{\Pr[M=\lfloor \fkappa_1 m\rfloor]}. \label{eq:ld3Mbound}
\end{align}
The ratios in the preceding expression are estimated using~\eqref{eq:Mpr}. In particular, Lemma~\ref{lemma:concentrate2} implies
\[
\frac1m \log \left(m^{-f_1} \Pr[\sM=\lfloor \fkappa_1 m \rfloor ]\, \expect[(m-\lfloor \fkappa_1 m \rfloor +f_1-1-\sS)_{f_1-1} \ind{\sS \leq m - \lfloor \fkappa_1 m \rfloor}] \right) \to 0,
\]
as $m \to \infty$. Similarly, $\sM$ being Poisson and monotonicity yield  
\[
\limsup_{m \to\infty} \frac1m \log \left(m^{-f_1} \Pr[\sM=\lfloor (\fkappa_1+\epsilon) m \rfloor ]\, (m +f_1-1)_{f_1-1} \right) < 0,
\]
and, for $\fkappa_1 > 0$,
\[
\limsup_{m \to\infty} \frac1m \log \left(m^{-f_1} \Pr[\sM=\lfloor (\fkappa_1-\epsilon) m \rfloor ]\, (m +f_1-1)_{f_1-1}  \right) < 0.
\]
Finally, the desired limit follows from the preceding limits, \eqref{eq:ld3Mbound} and~\eqref{eq:Mpr}.

\bibliographystyle{plain}
\bibliography{petarbib}

\begin{thebibliography}{10}

\bibitem{ADW13}
J.~Anselmi, B.~D'Auria, and N.~Walton.
\newblock Closed queueing networks under congestion: {N}onbottleneck
  independence and bottleneck convergence.
\newblock {\em Math. Oper. Res.}, 38(3):469--491, 2013.

\bibitem{BFL17}
S.~Banerjee, D.~Freund, and T.~Lykouris.
\newblock Pricing and optimization in shared vehicle systems: {A}n
  approximation framework.
\newblock Preprint, 2017.

\bibitem{BLW20}
S.~Benjaafar, H.~Liu, and S.~Wu.
\newblock Dimensioning on-demand vehicle sharing systems.
\newblock Preprint, 2020.

\bibitem{BBK99}
A.~Berger, L.~Bregman, and Y.~Kogan.
\newblock Bottleneck analysis in multiclass closed queueing networks and its
  application.
\newblock {\em Queueing Syst. Theory Appl.}, 31(3-4):217–237, 1999.

\bibitem{BDL19}
A.~Braverman, J.G. Dai, X.~Liu, and L.~Yin.
\newblock Empty-car routing in ridesharing systems.
\newblock {\em Oper. Res.}, 67(5):1437–1452, 2019.

\bibitem{CGY11}
M.~Chesarone-Cataldo, C.~Gu\'erin, J.H. Yu, R.~Wedlich-Soldner, L.~Blanchoin,
  and B.L.Goode.
\newblock The myosin passenger protein {Smy1} controls actin cable structure
  and dynamics by acting as a formin damper.
\newblock {\em Dev. Cell}, 21:217–230, 2011.

\bibitem{How11}
M.K. Gardner, M.~Zanic, C.~Gell, V.~Bormuth, and J.~Howard.
\newblock Depolymerizing kinesins {Kip3} and {MCAK} shape cellular microtubule
  architecture by differential control of catastrophe.
\newblock {\em Cell}, 147:1092–1103, 2011.

\bibitem{GoH12}
N.W. Goehring and A.A. Hyman.
\newblock Organelle growth control through limiting pools of cytoplasmic
  components.
\newblock {\em Curr. Biol.}, 22(9):R330--R339, 2012.

\bibitem{GVS13}
M.C. Good, M.D. Vahey, A.~Skandarajah, D.A. Fletcher, and R.~Heald.
\newblock Cytoplasmic volume modulates spindle size during embryogenesis.
\newblock {\em Science}, 342:856--860, 2013.

\bibitem{JSS14}
K.~Johnson, D.~Simchi-Levi, and P.~Sun.
\newblock Analyzing scrip systems.
\newblock {\em Oper. Res.}, 62(3):524–534, 2014.

\bibitem{KnT90}
C.~Knessl and C.~Tier.
\newblock Asymptotic expansion for large closed queuing networks.
\newblock {\em J. ACM}, 37(1):144--174, 1990.

\bibitem{Kog95}
Y.~Kogan.
\newblock Another approach to asymptotic expansions for large closed queueing
  networks.
\newblock {\em Oper. Res. Lett.}, 11(5):317--321, 1992.

\bibitem{KBD20}
L.v. Kreveld, O.~Boxma, J.-P. Dorsman, and M.~Mandjes.
\newblock Scaling limits for closed product-form queueing networks.
\newblock Preprint, 2020.

\bibitem{Mar16}
W.F. Marshall.
\newblock Cell geometry: {H}ow cells count and measure size.
\newblock {\em Annu. Rev. Biophys.}, 45:49–64, 2016.

\bibitem{Mar05}
W.F. Marshall, H.~Qin, M.R. Brenni, and J.L. Rosenbaum.
\newblock Flagellar length control system: {T}esting a simple model based on
  intraflagellar transport and turnover.
\newblock {\em Mol. Biol. Cell}, 16:270–278, 2005.

\bibitem{MMi82}
J.~Mc{K}enna and D.~Mitra.
\newblock Integral representations and asymptotic expansions for closed
  {M}arkovian queueing networks: {N}ormal usage.
\newblock {\em Bell Syst. Tech. J.}, 1982.

\bibitem{MiD11}
A.~Michelot and D.G. Drubin.
\newblock Building distinct actin filament networks in a common cytoplasm.
\newblock {\em Curr. Biol.}, 21:R560–R569, 2011.

\bibitem{MGJ16}
L.~Mohapatra, B.L. Goode, P.~Jelenkovi\'c, R.~Phillips, and J.~Kondev.
\newblock Design principles of length control of cytoskeletal structures.
\newblock {\em Annu. Rev. Biophys.}, 45:85--116, 2016.

\bibitem{MGK15}
L.~Mohapatra, B.L. Goode, and J.~Kondev.
\newblock Antenna mechanism of length control of actin cables.
\newblock {\em PLoS Comput. Biol.}, 11(6):e1004160, 2015.

\bibitem{MLH17}
L.~Mohapatra, T.J. Lagny, D.~Harbage, P.R.~Jelenkovi\' c, and J.~Kondev.
\newblock The limiting-pool mechanism fails to control the size of multiple
  organelles.
\newblock {\em Cell Syst.}, 4(5):559--567, 2017.

\bibitem{Pit79}
B.~Pittel.
\newblock Closed exponential networks of queues with saturation: {T}he
  {J}ackson-type stationary distribution and its asymptotic analysis.
\newblock {\em Math. Oper. Res.}, 4(4):357--378, 1979.

\bibitem{WeB15}
S.C. Weber and C.P. Brangwynne.
\newblock Inverse size scaling of the nucleolus by a concentration-dependent
  phase transition.
\newblock {\em Curr. Biol.}, 25:641--646, 2015.

\bibitem{WES95}
A.~Weiss and A.~Shwartz.
\newblock {\em Large Deviations for Performance Analysis: {Q}ueues,
  Communications, and Computing}.
\newblock New York: Chapman \& Hall, 1995.

\end{thebibliography}

\end{document}